\documentclass[12pt]{amsart}
\usepackage{amssymb}
\usepackage{mathtools}

\numberwithin{equation}{section}

\def\1#1{\overline{#1}}
\def\2#1{\widetilde{#1}}
\def\3#1{\widehat{#1}}
\def\4#1{\mathbb{#1}}
\def\5#1{\frak{#1}}
\def\6#1{{\mathcal{#1}}}

\newcommand{\UH}{\mathbb{H}}

\newcommand{\Hol}{{\sf Hol}}

\def\Re{{\sf Re}\,}
\def\Im{{\sf Im}\,}

\newcommand{\Real}{\mathbb{R}}
\newcommand{\Natural}{\mathbb{N}}
\newcommand{\Complex}{\mathbb{C}}
\newcommand{\anglim}{\angle\lim}

\newcommand{\clskip}{0.1em}

\newcommand{\clskipii}{0.033em}
\newcommand{\clskipa}{0.067em}
\newcommand{\closure}[1]{{\overline{#1}}}

\newcommand{\UDc}{\hskip-\clskipii{\overline{\hskip\clskipii\mathbb D\hskip-\clskip}\hskip\clskip}}
\newcommand{\ComplexE}{{\hskip\clskipa\overline{\vphantom{J^m}%
\hskip-\clskipa\mathbb C\hskip-\clskipa}\hskip\clskipa}}

\newcommand{\UD}{\mathbb{D}}

\newcommand{\Step}[2]{\begin{itemize}\item[{\bf Step~#1.}]{\it #2}\end{itemize}}
\newcommand{\step}[2]{\begin{itemize}\item[{\it Step~#1.}]{\it #2}\end{itemize}}
\newcommand{\case}[2]{\begin{itemize}\item[{\it\underline{Case~#1}:}]{\it #2}\end{itemize}}

\newcommand{\mcite}[1]{\csname b@#1\endcsname}
\newcommand{\UC}{\mathbb{T}}

\newcommand{\dAlg}{{\mathcal A}(\UD)}
\newcommand{\diam}{\mathsf{diam}}

\theoremstyle{theorem}
\newtheorem {result} {Theorem}
\def\id{{\sf id}}

\def\Re{{\sf Re}\,}
\def\Im{{\sf Im}\,}

\newtheorem{theorem}{Theorem}[section]
\newtheorem{lemma}[theorem]{Lemma}
\newtheorem{proposition}[theorem]{Proposition}
\newtheorem{corollary}[theorem]{Corollary}

\theoremstyle{definition}
\newtheorem{definition}[theorem]{Definition}

\theoremstyle{remark}
\newtheorem{remark}[theorem]{Remark}
\newtheorem{example}{Exapmle}
\numberwithin{equation}{section}

\newenvironment{mylist}{\begin{list}{}%
{\labelwidth=2em\leftmargin=\labelwidth\itemsep=.4ex plus.1ex
minus.1ex\topsep=.7ex plus.3ex
minus.2ex}%
\let\itm=\item\def\item[##1]{\itm[{\rm ##1}]}}{\end{list}}

\long\def\REM#1{\relax}

\def\bfit{\fontseries{bx}\fontshape{it}\selectfont}

\title[Angular and unrestricted limits of one-parameter semigroups]{Angular and unrestricted
limits of one-parameter semigroups in the unit disk}
\author[P. Gumenyuk]{Pavel Gumenyuk${}^\dag$}
\address{Dipartimento di Matematica, Universit\`a di Roma ``Tor Vergata", Via della Ricerca
Scientifica 1, 00133, Roma, Italia.} \email{gumenyuk@mat.uniroma2.it}

\thanks{${}^\dag$Partially supported by the FIRB grant Futuro in Ricerca ``Geometria Differenziale Complessa e Dinamica Olomorfa'' n. RBFR08B2HY}

\begin{document}

\maketitle \begin{abstract} We study local boundary behaviour of one-parameter
semigroups of holomorphic functions in the unit disk. Earlier under some
addition condition (the position of the Denjoy\,--\,Wolff point) it was shown
in~\cite{CMP2004} that  elements of one-parameter semigroups have
\textit{angular limits everywhere on the unit circle} and \textit{unrestricted
limits at all boundary fixed points}. We prove stronger versions of these
statements with no assumption on the position of the Denjoy\,--\,Wolff point.
In contrast to many other problems, in the question of existence for
unrestricted limits it appears to be more complicated to deal with the boundary
Denjoy\,--\,Wolff point (the case not covered in~\cite{CMP2004}) than with all
the other boundary fixed points of the semigroup.
\end{abstract}

\tableofcontents

\section{Introduction} One-parameter semigroups in the unit
disk~$\UD:\{z\in\Complex:\,|z|<1\}$ are classical objects of study in Complex
Analysis and can be defined as \textit{continuous homomorphisms} from the
additive semigroup~$\big([0,+\infty),+\big)$ of non-negative reals to the
topological semigroup $\Hol(\UD,\UD)$ consisting of all holomorphic
self-maps~$\phi:\UD\to\UD$ and endowed with operation of composition
$(\phi,\psi)\mapsto \psi\circ\phi$ and the topology induced by the locally
uniform convergence in~$\UD$. In other words, a \textit{one-parameter semigroup
in~$\UD$} is a family $(\phi_t)\in\Hol(\UD,\UD)$ satisfying the following
conditions:
\begin{enumerate}
\item[(i)] $\phi_0=\id_\UD$;

\item[(ii)] $\phi_{t+s}=\phi_t\circ\phi_s=\phi_s\circ\phi_t$ for any
$t,s\ge0$;

\item[(iii)] $\phi_t(z)\to z$ as $t\to+0$ for any $z\in\UD$.
\end{enumerate}
Due to the fact that~$\Hol(\UD,\UD)$ is a normal family in~$\UD$,
condition~(iii) expresses the continuity of the map $t\mapsto\phi_t$.

In a similar way one can define one-parameter semigroups in other
domains\footnote{It is worth to mention that one-parameter semigroups in a
domain~$D\subset\ComplexE$ constitute a very narrow class of objects unless $D$
is conformally equivalent to~$\UD$. Passing from~$\UD$ to another domain can
make sense when the geometry of the new domain suites the problem better,
\textsl{e.g.}, in the case of a boundary attracting fixed point (the boundary
Denjoy\,--\,Wolff point).}, \textsl{e.g.}, in the upper-half plane
$\UH_i:=\{z:\,\Im z>0\}$. In what follows we will omit the words ``in~$\UD$"
and specify the domain only in the rare cases when it is different from~$\UD$.

Interest to one-parameter semigroups comes from different areas. In the
Iteration Theory in~$\UD$ they appear as \textit{fractional iterates}, see,
\textsl{e.g.}, \cite{frac1958, frac1981, fracGor1991, fracGor2002}. In Operator
Theory, one-parameter semigroups in~$\UD$ have been extensively investigated in
connection with the study of one-parameter semigroups of~\textit{composition
operators}, see, \textsl{e.g.}, \cite{BP, SiskakisReview}. The
\textit{embedding problem} for time-homogenous stochastic branching processes
is also very much related to one-parameter semigroups, see, \textsl{e.g.},
\cite{emb1968,emb1993Dokl,emb1993MatSb}. Finally, one can consider this notion
as a special (autonomous) case of evolution families in~$\UD$ playing important
role in much celebrated Loewner Theory \cite{SemigroupsGor, BCM1,MFMS_ems}. It
is also worth to be mentioned that one-parameter semigroups lying in a given
subsemigroup $\mathsf S\subset\Hol(\UD,\UD)$ give useful information about the
infinitesimal structure of~$\mathsf S$~\cite{SemigroupsGor,fracGor1991}.

Not every element of~$\Hol(\UD,\UD)$ can be embedded into a one-parameter
semigroup. Elements of one-parameter semigroups enjoy some very specific nice
properties. For example, these functions are \textit{univalent} (see,
\textsl{e.g.}, \cite[Proposition~(1.4.6)]{Abate}). But especially brightly this
shows up in their boundary behaviour. In this paper we study mainly
\textit{local boundary behaviour} of one-parameter semigroups.

\subsection{Preliminaries}
Here we collect some fundamental results on one-parameter semigroups we use in
this paper.

First of all, in spite of the fact that in the definition one requires only
continuity of a one-parameter semigroup~$(\phi_t)$ w.r.t. the parameter~$t$,
the algebraic semigroup structure enhances regularity in~$t$. In fact, the map
$t\mapsto \phi_t(z)$ is smooth. Moreover, any one-parameter semigroup is a
semiflow of some holomorphic vector field \cite[Theorem(1.1)]{BP}, see also
\cite[\S3.2]{ShoikhetBook} or \cite[Theorem~(1.4.11)]{Abate}. More rigorously
these statements can be formulated in the following form.
\begin{result}\label{RES_inf-gen}
For any one-parameter semigroup~$(\phi_t)$ the limit
\begin{equation}
G(z):=\lim_{t\to+0}\frac{\phi_t(z)-z}{t},\quad z\in\UD,
\end{equation}
exists and $G$ is a holomorphic function in~$\UD$.

Moreover, for each $z\in\UD$, the function $[0,\infty)\ni t\mapsto
w(t):=\phi_t(z)\in\UD$ is the unique solution to the initial value problem
\begin{equation}\label{EQ_ODEaut}
\frac{dw(t)}{dt}=G\big(w(t)\big),\quad t\ge0,\quad w(0)=z.
\end{equation}
\end{result}
\begin{definition}\label{D_inf-gen}
The function~$G$ in Theorem~\ref{RES_inf-gen} is called the
\textit{infinitesimal generator} of the one-parameter semigroup~$(\phi_t)$.
\end{definition}
Clearly, not every holomorphic function in~$\UD$ is a generator of a
one-parameter semigroup. Berkson and Porta \cite{BP} obtained the following
very useful characterization of infinitesimal generators (see
also~\cite[Theorem(1.4.19)]{Abate}).

\begin{result}\label{RES-BP}
A function $G:\UD\to\Complex$ is an infinitesimal generator (of some
one-parameter semigroup in~$\UD$) if and only if it can be represented as
\begin{equation}\label{EQ-BP}
G(z)=(\tau-z)(1-\overline\tau\,z) p(z),
\end{equation}
where $\tau$ is a point of~$\UDc$ and $p\in\Hol(\UD,\Complex)$ satisfies the
condition $\Re p(z)\ge0$ for all $z\in\UD$.

The representation~\eqref{EQ-BP} is unique unless $G\equiv0$, in which case
(trivially) $p\equiv0$ and $\tau$ is any point in~$\UDc$.
\end{result}

\noindent\textbf{Assumption.} In what follows we assume that all one-parameter
semigroups~$(\phi_t)$ we consider are non-trivial, \textsl{i.e.} at least one
of $\phi_t$'s is different from~$\id_\UD$. Except for the case of elliptic
automorphisms\footnote{\textsl{I.e.}, the case when all the functions~$\phi_t$
are automorphisms of~$\UD$ with a common fixed point~$\tau\in\UD$.}, this
condition in fact implies (see, \textsl{e.g.}, \cite[p.\,108--109]{Abate}) that
$\phi_t\neq\id_\UD$ for all $t>0$ and that the infinitesimal generator
$G\not\equiv0$.

It is an immediate consequence of the Schwarz Lemma that a self-map
$\phi\in\Hol(\UD,\UD)\setminus\{\id_\UD\}$ can have at most one fixed point
in~$\UD$. However, there can be much more so-called \textit{boundary fixed
points}.
\begin{definition}
Let $\phi\in\Hol(\UD,\UD)$ and $\sigma\in\UC:=\partial\UD$. The point $\sigma$
is called a \textit{contact point} of~$\phi$ if the angular limit
$\phi(\sigma):=\anglim_{z\to\sigma}\phi(z)$ exists and belongs to~$\UC$. If in
addition, $\varphi(\sigma)=\sigma$, then $\sigma$ is said to be a
\textit{boundary fixed point} of~$\phi$.
\end{definition}

It is known that if $\sigma\in\UC$ is a contact point
of~$\phi\in\Hol(\UD,\UD)$, then the angular limit
$\phi'(\sigma):=\anglim_{z\to\sigma}\big(\phi(z)-\phi(\sigma)\big)/(z-\sigma)$,
referred to as the \textit{angular derivative} of~$\phi$ at~$\sigma$, exists,
\textit{finite or infinite}, see,
\textsl{e.g.},~\cite[Proposition~4.13]{Pommerenke-II}.
\begin{definition}
A contact point (resp., boundary fixed point) $\sigma\in\UC$ of a
self-map~$\phi\in\Hol(\UD,\UD)$ is said to be \textit{regular} if
$\phi'(\sigma)\neq\infty$.
\end{definition}
\begin{remark}\label{RM_Julia-Lemma}
By the classical Julia Lemma (see, \textsl{e.g.}, \cite[Chapter~1,
Exercises~6,\,7]{Garnet}) the following two statements are equivalent:
\begin{itemize}
\item[(a)] $\phi\in\Hol(\UD,\UD)$ has a regular contact point
at~$\sigma\in\UC$;
\item[(b)] the dilation
$$
\alpha_\phi(\sigma):=\liminf_{\UD\ni z\to\sigma}\frac{1-|\phi(z)|}{1-|z|}
$$
is finite.
\end{itemize}
Moreover, if the above conditions are fulfilled, then
$\phi'(\sigma)=\overline\sigma\phi(\sigma)\alpha_\phi(\sigma)$.\REM{ and for
all $z\in\UD$,
$$
\frac{|\phi(z)-\phi(\sigma)|^2}{1-|\phi(z)|^2}\le|\phi'(\sigma)|\frac{|z-\sigma|^2}{1-|z|^2}.
$$}
\end{remark}

The following statement is fundamental for the study of $\Hol(\UD,\UD)$, see,
\textsl{e.g.}, \cite[\S\S1.3,\,1.4]{ShoikhetBook}.
\begin{result}[Denjoy\,--\,Wolff Theorem]
Let $\phi\in\Hol(\UD,\UD)\setminus\{\id_\UD\}$. Then there exists a unique
(boundary) fixed point~$\tau\in\UDc$ such that\,\footnote{If
$\tau\in\partial\UD$, then $\phi'(\tau)$ stands, of course, for the angular
derivative at~$\tau$.} $|\phi'(\tau)|\le 1$. Moreover, if $\phi$ is not an
elliptic automorphism of~$\UD$, then the iterates $\phi^{\circ n}\to\tau$
locally uniformly in~$\UD$ as $n\to+\infty$.
\end{result}
The point~$\tau$ in the above theorem is called the \textit{Denjoy\,--\,Wolff
point} of~$\phi$ (abbreviated, ``DW-point"). It is known (see, \textsl{e.g.},
\cite[Corollary(1.4.18), Theorem(1.4.19)]{Abate}) that for a one-parameter
semigroup~$(\phi_t)$ the functions $\phi_t$, $t>0$, share the same DW-point,
which coincides with the point~$\tau$ in the Berkson\,--\,Porta
formula~\eqref{EQ-BP} for the infinitesimal generator~$G$ of~$(\phi_t)$. This
point is called the \textit{Denjoy\,--\,Wolff point of the one-parameter
semigroup~$(\phi_t)$}. The following theorem allows also to define boundary
fixed points and boundary regular fixed points of one-parameter semigroups.
\begin{result}[\protect{\cite[Theorems~1\,and~5]{CMP2004}; \cite[Lemmas~1\,and~3]{CMP2006}}]
\label{RES_common-set-of-BFPs} Let $(\phi_t)$ be a one-parameter semigroup
in~$\UD$ and $\sigma\in\UC$. Then:
\begin{itemize}
\item[(i)] $\sigma$ is a fixed point of $\phi_t$ for some~$t>0$ if and only
if it is a fixed point of $\phi_t$
for all~$t>0$;

\item[(ii)] $\sigma$ is a boundary regular fixed point of~$\phi_t$ for
some $t>0$ if and only if it is a boundary regular fixed point of $\phi_t$ for
all~$t>0$.
\end{itemize}
\end{result}

\begin{remark}\label{RM_cont-DW-theorem}
The Denjoy\,--\,Wolff Theorem implies easily, see,
\textsl{e.g.},~\cite[Theorem(1.4.17)]{Abate}, that similar to the case of
discrete iteration, any one-parameter semigroup~$(\phi_t)$ such that $\phi_t$'s
are not elliptic automorphisms of~$\UD$ for~$t>0$, converges locally uniformly
in~$\UD$ to its Denjoy\,--\,Wolff point as $t\to+\infty$.
\end{remark}

Besides infinitesimal representation given by Theorem~\ref{RES_inf-gen},
one-parameter semigroups can be represented by means of the so-called
\textit{linearization models}. By a linearization model for a one-parameter
semigroup~$(\phi_t)$ we mean a three-tuple $\big(h,\Omega,\mathcal T\big)$,
where $\mathcal T=(\mathcal L_t)$ is a one-parameter semigroup in~$\ComplexE$
consisting of M\"obius transformations, $\Omega\subset\ComplexE$ is a simply
connected domain with $|\ComplexE\setminus\Omega|>1$, and $h$ is a conformal
mapping of~$\UD$ onto $\Omega$ such that $\mathcal L_t(\Omega)\subset\Omega$
and $\mathcal L_t\circ h=h\circ\phi_t$ for all $t\ge0$. The choice of a
``standard" linearization model depends on whether the DW-point lies in~$\UD$
or on its boundary.
\begin{result}[see, \textsl{e.g.}, \protect{\cite[Theorems\,(1.4.22),\,(1.4.23)]{Abate}}]
Let be $(\phi_t)$ be a one-parameter semigroup in~$\UD$ and $\tau$ its
Denjoy\,--\,Wolff point. Then:
\begin{itemize}
\item[(A)] If $\tau\in\UD$, then there exists a univalent holomorphic function
$h:\UD\to\Complex$ with $h(\tau)=0$ satisfying the Schr\"oder functional
equation
$$
h\circ \phi_t=\phi_t'(\tau)h \quad\text{for all~$t\ge0$}.
$$ Such a function~$h$ is unique up to multiplication by a complex constant
$h\mapsto c h$, where $c\in\Complex^*$.

\item[(B)] If $\tau\in\UC$, then there exists a univalent holomorphic function
$h:\UD\to\Complex$ satisfying the Abel functional equation
$$
h\circ\phi_t=h+t\quad\text{for all~$t\ge0$}.
$$ Such a function~$h$ is unique up to a translation $h\mapsto h+c$, where $c\in\Complex$.
\end{itemize}
\end{result}
The function $h$ in the above theorem is called the \textit{K\oe{}nigs function}
of the one-parameter semigroup~$(\phi_t)$. Usually, to fix the unique solution
to the Schr\"oder and Abel equations, one assumes that $h'(\tau)=1$ or $h(0)=0$
in the former and latter cases, respectively. However, for our purposes it will
be more convenient not to impose this normalization in the case of boundary
DW-point.

\subsection{Main results.}

Although  one-parameter semigroups in~$\UD$ constitute a classical topic and
the study of holomorphic self-maps of~$\UD$ suggests looking for angular limits
on the boundary, it was not realized before~\cite{CMP2004} that
\textit{elements of one-parameter semigroups have angular limits everywhere on
the unit circle}. This fact does not seem to be widely known: in that paper it
was stated only in a proof (the proof of Theorem~5) and only for the case of
boundary DW-point. We note that it is, in fact, enough to consider this case,
as one can see using the idea from~\cite[Proof of Theorem~3.3]{Tiz}. The
corresponding auxiliary statement is proved in section~\ref{S_lifting}, which
allows us to concentrate in what follows on the case of boundary DW-point.

First of our main results, Theorem~\ref{TH_anglims} is a ``uniform version'' of
the fact stated above. As usual we denote by $\phi_t(\sigma)$, $\sigma\in\UC$,
the angular limit of~$\phi_t$ at~$\sigma$. We show that for each Stolz angle
$S$ with vertex~$\sigma\in\UC$ the convergence $\phi_t(z)\to\phi_t(\sigma)$  as
$S\ni z\to\sigma$ is \textit{locally uniform in~$t$}.

Further in Proposition~\ref{PR_cont-in-t} we show that a one-parameter
semigroup~$(\phi_t)$ being considered as a family of maps $[0,+\infty)\ni
t\mapsto \phi_t(z)\in\UDc$ parameterized by $z\in\UDc$, is uniformly
continuous. In particular,  \textit{for each~$\sigma\in\UC$ the
trajectory~$t\mapsto \phi_t(\sigma)$ is continuous}. Moreover, as a byproduct,
in section~\ref{SS_rem-loc-beh} we will see (Remark~\ref{RM_conv-to-DW}) that
\textit{if the DW-point~$\tau$ belongs to~$\partial\UD$, then $\sigma\in\UC$ is
either a boundary fixed point of~$(\phi_t)$, or $\phi_t(\sigma)\to\tau$ as
$t\to+\infty$.} The analogous statement for~$\tau\in\UD$ follows readily
from~\cite[Theorem~4]{CMP2004}.

Despite of the above remarkable facts, the extension of $\phi_t$ by angular
limits is not necessarily continuous on~$\UC$. In other words, the unrestricted
limits do not need to exists everywhere on~$\UC$. As a ``compensation'', the
\textit{unrestricted limits still do exist at all (regular and non-regular)
boundary fixed points}. For the first time this was proved
in~\cite[Corollary~3]{CMP2004} for the case of interior DW-point. We prove a
``uniform version'' of this statement for the boundary DW-point, see
Theorem~\ref{TH_cont-at-fixed-points}, which automatically extends to the
interior case due to Proposition~\ref{PR_lifting}. For all repelling fixed
points on~$\UC$ we could employ essentially the same idea as in~\cite{CMP2004}:
the key point is to use the translational invariance of~$\Omega:=h(\UD)$ in
order to prove that the K\oe{}nigs function~$h$ has unrestricted limits at all
boundary repelling fixed points. However, the analogous statement for the
DW-point does not hold, and we had to give an independent proof for this
distinguished fixed point. This aspect is really new in the boundary DW-point
case.

The proof of Theorem~\ref{TH_cont-at-fixed-points} involves several more
technical results, \textsl{e.g.}, Propositions~\ref{Pr_repelling_iff},
\ref{PR_if_h_cont_then_phi_cont}, \ref{PR_continuous_h}, and \ref{PR_hyperbol},
which might be of some interest for specialists.

In section~\ref{SS_rem-loc-beh} we consider three examples. The first two of
them are related to the local dynamical behaviour of $\phi_t:\UDc\to\UDc$ in a
neighbourhood of the boundary DW-point. The third example shows that
Theorem~\ref{TH_cont-at-fixed-points} cannot be extended to contact points of
one-parameter semigroups.

\begin{remark}
One might ask if the unrestricted limits exist also \textit{at all contact
points} of the one-parameter semigroup. The answer is ``no'', see
Remark~\ref{RM_no-for-contract-points}. Going in another direction one might
also ask if at every boundary fixed point there exists the unrestricted limit
of the derivative~$\phi_t'$ and/or the ``unrestricted derivative'':
$$
\lim_{\UD\ni z\to\sigma}\phi_t'(z),\qquad \lim_{\UD\ni
z\to\sigma}\frac{\phi_t(z)-\phi_t(\sigma)}{z-\sigma}.
$$
The answer is again ``no''. Both unrestricted limits above fail to exist if our
boundary regular fixed point~$\sigma$ is not isolated, \textsl{i.e.} if it is a
limit of a sequence of boundary fixed points different from~$\sigma$. For
repelling fixed points such examples can be obtained by modifying the
construction given in~\cite[p.\,260]{AnaFlows}. The non-isolated hyperbolic
DW-point appears in Example~\ref{EX1}.
\end{remark}

The last section of this paper, section~\ref{S_EF-diskalgebra}, is devoted to a
question concerning boundary behaviour of the non-autonomous generalization of
one-parameter semigroup, the so-called \textit{evolution families} in~$\UD$. It
is known that any univalent~$\varphi\in\Hol(\UD,\UD)$ can be embedded into an
evolution family. Therefore, we cannot expect any results for evolution
families similar to the above results for one-parameter semigroups. However,
there is still the question whether the algebraic structure of evolution family
affects the relationships between various analytic properties, in particular
those of regularity on the boundary. We prove
(Proposition~\ref{PR_EF-diskalgebra}) that if all the elements of an evolution
family~$(\varphi_{s,t})$ are continuous in~$\UDc$, then the map $t\mapsto
\varphi_{s,t}$ is continuous w.r.t. the supremum norm for any fixed~$s\ge0$.
The proof is based on an extended version of the No-Koebe-Arcs Theorem, see,
\textsl{e.g.}, \cite[Theorem~9.2]{Pommerenke}.

\section{Lifting one-parameter semigroups with the interior DW-point}\label{S_lifting}
Obviously, using M\"obius transformations of~$\UD$ one can assume that the
DW-point of a given one-parameter semigroup is either $\tau=0$ or $\tau=1$. In
fact, we can further reduce, up to some extend, the case of interior DW-point
($\tau=0$) to the case of boundary DW-point ($\tau=1$). This is the meaning of
the following elementary proposition\footnote{The same idea in a bit different
context was used in~\cite{Tiz}.}. In what follows for
$a\in\Complex^*:=\Complex\setminus\{0\}$, we denote
$$\UH_a:=\{z\in\Complex:\Re(\overline az)>0\}.$$
\begin{proposition}\label{PR_lifting}
Let $(\phi_t)$ be a  non-trivial one-parameter semigroup in~$\Hol(\UD,\UD)$
with the DW-point $\tau=0$. Let $h$ be its K\oe{}nigs functions and $G$ its
infinitesimal generator. Then there exists a unique one-parametric semigroup
$(\tilde\phi_t)$ in $\Hol(\UH_1,\UH_1)$ with the DW-point~$\tilde\tau=\infty$
such that for all $t\ge0$ and all $\tilde z\in\UH_1$ we have
\begin{equation}\label{EQ_phi-phi_tilde}
\exp\big(-\tilde\phi_t(\tilde z)\big)=\phi_t\big(\exp(-\tilde z)\big).
\end{equation}
Moreover, the K\oe{}nigs function~$\tilde h_0$ and the infinitesimal
generator~$\tilde G$ of the one-parameter semigroup~$(\tilde \phi_t)$ are given
by
\begin{equation}\label{EQ_phi-phi_tilde_h_G}
\tilde h_0(\tilde z)=-\frac{\tilde h(\tilde z)}{G'(0)},\quad \tilde G(\tilde
z)=-\frac{G\big(\exp(-\tilde z)\big)}{\exp(-\tilde z)}\qquad\text{for all
$\tilde z\in\UH_1$,}
\end{equation}
where $\tilde h:\UH_1\to\Complex$ is a holomorphic lifting of $\UH_1\ni\tilde
z\mapsto h\big(\exp(-\tilde z)\big)\in\Complex^*$ w.r.t. the covering map
$\Complex\ni\tilde w\mapsto \exp(-\tilde w)\in\Complex^*$.
\end{proposition}
\begin{proof}
For the proof of the uniqueness and of formulas~\eqref{EQ_phi-phi_tilde_h_G},
we first assume that there exists a one-parameter semigroup~$(\tilde \phi_t)$
satisfying~\eqref{EQ_phi-phi_tilde}. Since $h$ is univalent and $h(0)=0$, we
have $h(\UD^*)\subset\Complex^*$. According to the Monodromy Theorem there
exists a holomorphic lifting $\tilde h:\UH_1\to\Complex$ of $\UH_1\ni\tilde
z\mapsto h\big(\exp(-\tilde z)\big)\in\Complex^*$ w.r.t. the covering map
$\Complex\ni\tilde w\mapsto \exp(-\tilde w)\in\Complex^*$. This means that
\begin{equation}\label{EQ_tilde_h}
\exp\big(-\tilde h(\tilde z)\big)=h\big(\exp(-\tilde z)\big)\quad\text{for all
$\tilde z\in\UH_1$.}
\end{equation}
 The K\oe{}nigs function~$h$ of $(\phi_t)$ satisfies
the Schr\"oder functional equation
\begin{equation}\label{EQ_Schroeder}
h(\phi_t(z))=e^{\lambda t}h(z),\quad\text{for all $t\ge0$ and all $z\in\UD$,}
\end{equation}
 where $\lambda:=G'(0)$.
 Combining the latter two equalities
with~\eqref{EQ_phi-phi_tilde} one easily obtains
$$
\exp\big[-\tilde h\big(\tilde \phi_t(\tilde z)\big)\big]=\exp\big[\lambda
t-\tilde h(\tilde z)\big]\quad\text{for all $\tilde z\in\UH_1$ and all
$t\ge0$.}
$$
Taking into account that for any fixed~$\tilde z\in\UH_1$, $[0,+\infty)\ni
t\mapsto\tilde \phi_t(\tilde z)$ is continuous and equals~$\tilde z$
when~$t=0$, we conclude form the above equality that for all $t\ge0$ the
function $\tilde h_0:=-\tilde h/\lambda$ satisfies
\begin{equation}\label{EQ_for_tilde _h}
\tilde h_0\circ\tilde \phi_t=t+\tilde h_0.
\end{equation}

Differentiating~\eqref{EQ_for_tilde _h} w.r.t. $t$ one obtains $\tilde
h'_0=1/\tilde G$, while from~\eqref{EQ_Schroeder} it follows in a similar way
that $h'/h=\lambda/G$. Now combining these two equalities
and~\eqref{EQ_tilde_h}, we deduce the second formula
in~\eqref{EQ_phi-phi_tilde_h_G}. In particular, this proves the uniqueness of
the one-parameter semigroup~$(\phi_t)$, because it is defined uniquely by its
infinitesimal generator.

Furthermore, according to the Porta\,--\,Berkson formula~\eqref{EQ-BP}, $G$ is
of the form $G(z)=-zp(z)$, where $p\in\Hol(\UD,\Complex)$ with $\Re p(z)\ge0$
for all $z\in\UD$. Therefore,
\begin{equation}\label{EQ_ReGtilde}
\Re \tilde G(\tilde z)\ge0\quad\text{ for all $\tilde z\in\UH_1$.}
\end{equation}
Then, taking into account that $\tilde h'_0=1/\tilde G$ we may conclude with
the help of the Noshiro\,--\,Warschawski Theorem (see, \textsl{e.g.},
\cite[p.\,47]{Duren}) that $\tilde h_0$ is univalent in~$\UH_1$. Therefore,
this function is a K\oe{}nigs function of the semigroup~$(\tilde \phi_t)$. This
proves the first formula in~\eqref{EQ_phi-phi_tilde_h_G}. Besides that, it
follows from~\eqref{EQ_ReGtilde}, again with the help of the Berkson\,--\,Porta
formula, that any one-parameter semigroup $(\tilde \phi_t)$ in the
half-plane~$\UH_1$ satisfying~\eqref{EQ_phi-phi_tilde} must have the DW-point
at~$\infty$.

It remains to prove that such one-parameter semigroup~$(\tilde \phi_t)$ indeed
exists.
As the above argument shows, the fact that $G$ is a generator of a
one-parameter semigroup in~$\UD$ with the DW-point at~$\tau=0$ implies,
according to the Berkson\,--\,Porta formula, that the function $\tilde
G:\UH_1\to\Complex$ defined by the second formula
in~\eqref{EQ_phi-phi_tilde_h_G} is a generator of some one-parameter
semigroup~$(\tilde \phi_t)$ in~$\UH_1$ with the DW-point at~$\infty$. We claim
that this semigroup satisfies~\eqref{EQ_phi-phi_tilde}. Indeed, for any $\tilde
z\in\UH_1$,
\begin{multline*}
\frac{d}{dt}\exp\big(-\tilde \phi_t(\tilde z)\big)=-\exp\big(-\tilde
\phi_t(\tilde z)\big)\tilde G\big(\phi_t(\tilde z)\big)=G\big(\exp\big(-\tilde
\phi_t(\tilde z)\big)\big),\quad t\ge0,\\\exp\big(-\tilde \phi_t(\tilde
z)\big)|_{t=0}=z:=\exp(-\tilde z\big),
\end{multline*}
and therefore, by the uniqueness of the solution to the initial value problem
for ${dw/dt=G(w)}$, we have $\exp\big(-\tilde \phi_t(\tilde z)\big)=\phi_t(z)$
for all $t\ge0$, which proves~\eqref{EQ_phi-phi_tilde}.
\end{proof}

\section{Angular limits of one-parameter semigroups}
\subsection{Statement of results}
First of all we would like to formulate some general statements on the boundary
behaviour of one-parametric semigroups. Possibility to embed a holomorphic
self-map $\varphi$ of~$\UD$ into a one-parameter semigroup is a quite strong
condition. For example, it is well-known that the elements of
one-parameter semigroups are univalent in~$\UD$. Another, less elementary fact
is that these functions  must have angular limits at all points on~$\UC$. This
was proved in~\cite[p.\,479, proof of Theorem~5]{CMP2004} for the case of a
semigroup with the boundary DW-point. Here we prove a bit stronger statement
both for the interior and boundary DW-point.

\begin{theorem}\label{TH_anglims}
Let $(\phi_t)$ be a one-parameter semigroup. Then for any $t\ge0$ and any
$\sigma\in\UC$ there exists the angular limit
$\phi_t(\sigma):=\angle\lim_{z\to\sigma}\phi_t(z)$. Moreover, for each
$\sigma\in\UC$ and each Stolz angle $S$ with vertex at~$\sigma$ the convergence
$\phi_t(z)\to\phi_t(\sigma)$ as $S\ni z\to\sigma$ is locally uniform in
$t\in[0,+\infty)$.
\end{theorem}

Using the above theorem, we extend elements of the semigroup to the unit
circle~$\UC$. Suppressing the language in the same manner as in the statement
of Theorem~\ref{TH_anglims}, we will denote this extension again by~$\phi_t$.
Of course $\phi_t$'s do not need to be continuous w.r.t.\,$z$ on~$\UC$.
However, we will show that $t\mapsto\phi_t(z)$ is continuous in~$t$ for
any~$z\in\UDc$.

\begin{proposition}\label{PR_cont-in-t}
Let $(\phi_t)$ be a one-parameter semigroup. With~$\phi_t$, $t\ge0$, being
extended to~$\UC$ as in Theorem~\ref{TH_anglims}, the family of functions
$$\Big([0,+\infty)\ni t\mapsto\phi_t(z)\Big)_{z\in\UDc}$$
is uniformly equicontinuous.
\end{proposition}

The proofs are given below.


\subsection{Boundary behaviour of the K\oe{}nigs function. Proof of Proposition~\ref{PR_cont-in-t} and Theorem~\ref{TH_anglims}}\label{SS_PR_cont}
In what follows we will make use of one general statement concerning conformal
mappings of the disk. Denote by $\diam_U E$, where $U\subset\ComplexE$ is a
domain and $E\subset U$, the diameter of a set $E$ w.r.t. the standard
Poincar\'e metric of constant curvature on~$U$. Let $D$ be any domain
of~$\ComplexE$. For $w_1,w_2\in D$ denote
\begin{equation}\label{EQ_dist}
d_D(w_1,w_2):=\inf\big\{\diam_{\ComplexE}\,\Gamma:\,\text{$\Gamma\subset D$ is a Jordan arc joining $w_1$
and~$w_2$}\big\}.
\end{equation}
It is easy to see that $d_D$ is a distance function in~$D$.
\begin{proposition}\label{PR_unif-cont-inverse}
Let $f:\UD\to D$ be a conformal mapping onto a domain~$D\subset\ComplexE$. Then
for any $\varepsilon>0$ there exists $\delta>0$ such that if $K\subset\UD$ is a
pathwise connected set and $\diam_{\ComplexE} K<\delta$, then $\diam_\Complex
f^{-1}(K)<\varepsilon$. In particular, the inverse mapping $f^{-1}:D\to \UD$ is
uniformly continuous w.r.t. the distance~$d_D$ in $D$ and the Euclidean
distance in~$\UD$.
\end{proposition}
This proposition follows easily from Bagemihl\,--\,Seidel's version of the
No-Koebe-Arcs Theorem (see,~\textsl{e.g.},~\cite[Corollary~9.1
on~p.\,267]{Pommerenke}), taken into account that any conformal map is a normal
function (see,~\textsl{e.g.},~\cite[Lemma~9.3 on~p.\,262]{Pommerenke}).

Now we turn to the proof of Theorem~\ref{TH_anglims}. To this end we have to
study first the boundary behaviour of the K\oe{}nigs function.

It is known, see, \textsl{e.g.}, \cite[\S3.6]{Pommerenke-II}, that starlike
functions have angular limits\footnote{Note that for univalent functions, and
more generally, for all normal functions in~$\UD$, the angular limit at a given
point on~$\UC=\partial\UD$ exists if and only if the radial limit at this point
exists, see \textsl{e.g.} \cite[\S9.1, Lemma~9.3, Theorem~9.3]{Pommerenke}.},
finite or infinite, at every point on the unit circle. This is also the case
for the more general class of spiral-like
functions~\cite[Theorem~3.2]{KimSugawa}, which serve as K\oe{}nigs functions of
one-parameter semigroups with the interior DW-point. The proof of the analogous
result for K\oe{}nigs functions of one-parameter semigroups with the boundary
DW-point is very similar. We demonstrate here this proof only in order to make
the exposition more self-contained.
\begin{proposition}\label{PR_Re_h}
Let $(\phi_t)$ be a one-parameter semigroup in~$\UD$ with the DW-point $\tau=1$
and $h$ its K\oe{}nigs function. Then
$$\forall\sigma\in\UC\quad\exists~\angle\lim_{z\to\sigma}h(z)\in\ComplexE.$$ Moreover, if
$\sigma\in\UC\setminus\{\tau=1\}$, then
\begin{equation}\label{EQ_Re-h-to_-infty}
\limsup_{z\to\sigma} \Re h(z)<+\infty.
\end{equation}
\end{proposition}
Before proving the above proposition let us make some comments.
\begin{remark}\label{RM_h-to-infty=fixed}
Using the correspondence between the boundary accessible points of $\Omega$ and the
unit circle induced by $h$ (see, \textsl{e.g.}, \cite[Theorem~1 in Chapter~2,
\S3]{Goluzin}) it easy to see that $h(r\sigma)\to\infty$ as $r\to1-0$,
where~$\sigma\in\UC$, if and only if~$\sigma$ is a boundary fixed point
of~$(\phi_t)$, which in principle \textit{can coincide} with the DW-point. We
are able prove a bit more:
\end{remark}

\begin{proposition}\label{Pr_repelling_iff}
Let $(\phi_t)$ be a one-parameter semigroup in~$\UD$ with the DW-point
$\tau\in\UC$ and $h$ its K\oe{}nigs function. Let~$\sigma\in\UC$. Then the
unrestricted limit $\lim_{\UD\ni z\to\sigma}\Im h(z)$ exists finitely if and
only if $\sigma$ is not a boundary regular fixed point of~$(\phi_t)$.

Moreover, the following statements are equivalent:
\begin{itemize}
\item[(i)] $\sigma$ is a boundary fixed point of~$(\phi_t)$ other than its
DW-point;
\item[(ii)] $\Re h(r\sigma)\to-\infty$ as $r\to1-0$;
\item[(iii)] $\lim_{\UD\ni z\to\sigma} \Re h(z)=-\infty$.
\end{itemize}
\end{proposition}
\noindent We do not use Proposition~\ref{Pr_repelling_iff} in this section. Its proof will be given in section~\ref{SS_Koenigs}.

\begin{remark}\label{RM_H}
Let $(\phi_t)$ be a one-parameter semigroup in~$\UD$ with the DW-point $\tau=1$ and let $\sigma\in\UC\setminus\{\tau=1\}$.
The function
$F_\sigma(z):=(1+z)/(1-z)-(1+\sigma)/(1-\sigma)$ is the conformal mapping
of~$\UD$ onto~$\UH_1$ sending the DW-point~$\tau$ and the point~$\sigma$
to~$\infty$ and to the origin, respectively. The family~$(\Phi_t)_{t\ge0}$
defined by $\Phi_t:=F_\sigma\circ\phi_t\circ F_\sigma^{-1}$ for all~$t>0$ is a
one-parameter semigroup in~$\UH_1$ with the DW-point at~$\infty$. Its K\oe{}nigs
function $H:=h\circ F_\sigma^{-1}$ has the property that $\Re H'(z)\ge0$ for
all $z\in\UH_1$. In fact, it is easy to see that $\Re H'(z)>0$ for all
$z\in\UH$ unless $h$ is linear-fractional mapping and all $\phi_t$'s are
automorphisms of~$\UD$.
\end{remark}
Let us now prove one auxiliary statement.
\begin{lemma}\label{LM_HRe_prime_positive}
Let $H\in\Hol(\UH_1,\Complex)$ and $\Re H'(z)>0$ for all $z\in\UH_1$. Then
there exist constants $\alpha\in\Real$, $\beta\ge0$, and a bounded positive
Borel measure~$\mu$ on~$\Real$ such that for all~$z\in\UD$,
\begin{align}\label{EQ_ReprHprime}
&H'(z)=i\alpha+\beta z+\int_\Real\frac{1+itz}{z+it}\,
d\mu(t),\\\label{EQ_ReprH}
&H(z)=H(1)+i\alpha(z-1)+\frac{\beta}{2}(z^2-1)+\int_\Real\left[
it(z-1)+(1+t^2)\log\frac{z+it}{1+it}\right]d\mu(t),
\end{align}
where the branch of the logarithm in the second formula is chosen so that it
vanishes at~$z=1$.
\end{lemma}
\begin{proof}
The function $p(\zeta):=iH'(-i \zeta)$ is a holomorphic self-mapping
of~$\UH_i$. Therefore, $p$ admits the Nevanlinna representation (see,
\textsl{e.g.}, \cite[Vol.II,\,p.\,7]{AkhGlaz})
$$
p(\zeta)=\alpha+\beta
\zeta+\int_\Real\frac{1+t\zeta}{t-\zeta}\,d\mu(t)\qquad\text{for all
$\zeta\in\UH_i$},
$$
where $\alpha\in\Real$, $\beta\ge0$, and $\mu$ is a bounded positive Borel
measure on~$\Real$. Now replacing $\alpha$ by~$-\alpha$ we immediately
obtain~\eqref{EQ_ReprHprime}.

To deduce~\eqref{EQ_ReprH} we notice that for all $z\in\UH_1$ and $t\in\Real$
the integrand $A(z,t)$ in~\eqref{EQ_ReprHprime} satisfies
$$
|A(z,t)|\le\frac{1+2|z|^2}{\Re z}.
$$
(To check this inequality consider separately the cases $|t|\le2|z|$ and
$|t|\ge2|z|$.) Since the measure~$\mu$ is bounded, this allows us to
integrate~\eqref{EQ_ReprHprime} on the segment~$[1,z]$ using the Fubini
Theorem. This immediately leads to formula~\eqref{EQ_ReprH}. The proof is
complete.
\end{proof}

\begin{proof}[{\bfit Proof of Proposition~\ref{PR_Re_h}}]\mbox{~}
\step1{First let us prove the existence of the angular limits of~$h$ at every
point on~$\UC$.} We note that for $\sigma=\tau$ this statement is well-known.
Indeed, it follows from the fact for any $z_0\in\UD$ the trajectory
$[0,+\infty)\ni t\mapsto \gamma(t):=\phi_t(z_0)$ tends, as $t\to+\infty$, to
the DW-point~$\tau$, see Remark~\ref{RM_cont-DW-theorem}. Since $h$ satisfies
the Abel equation, we have $(h\circ\gamma)(t)=h(z_0)+t$ for all~$t\ge0$.
Therefore, $h(z)\to\infty$ as $z\to\tau$ along the curve~$\gamma$,
\textsl{i.e.}, $\infty$ is an asymptotic value of~$h$ at~$\tau$. Since $h$ is
univalent, it is normal in~$\UD$ (see,~\textsl{e.g.},~\cite[Lemma~9.3
on~p.\,262]{Pommerenke}). It follows (see,~\textsl{e.g.},~\cite[Theorem~9.3
on~p.\,268]{Pommerenke}) that the angular limit of~$h$ at~$\sigma=\tau$ exists
and equals~$\infty$.

Now we can assume that~$\sigma\neq\tau$. By the above argument it sufficient to
show that the function~$h$ has an asymptotic value at~$\sigma$. To this end,
according to Remark~\ref{RM_H}, we only have to prove that if
$H:\UH_1\to\Complex$ is holomorphic and $\Re H'>0$ in~$\UH_1$, then there
exists the limit of~$H(x)$ as~$(0,1)\ni x\to0$. Note that $\lim_{(0,1)\ni
x\to0}\Re H(x)$ exists (finite or infinite) because $(d/dx) \Re H(x)=\Re
H'(x)>0$. We claim that
\begin{equation}\label{EQ_limImH}
\text{the limit}\qquad \lim_{(0,1)\ni x\to0}\Im H(x)\qquad\text{exists
finitely.}
\end{equation}
Using representation~\eqref{EQ_ReprH} given in
Lemma~\ref{LM_HRe_prime_positive}, for all $x\in(0,1)$ we obtain
\begin{multline}\label{EQ_ReprImH}
\Im H(x)=\Im
H(1)+\alpha(x-1)+\frac{\beta}{2}(x^2-1)\\+\int_\Real\left[(x-1)t+(1+t^2)\left(\arctan\frac{t}{x}
- \arctan t\right)\right]\,d\mu(t).
\end{multline}
Note that the integrand~$B_x(t)$ in the above formula tends pointwise to
\begin{align*}
B_0(t)&:=-t+(1+t^2)\left(\frac\pi2\mathop{\mathrm{sgn}} t - \arctan
t\right)=-t+(1+t^2)\arctan\frac1t,\quad t\neq0,\\
B_0(0)&:=0,
\end{align*}
as $x\to+0$. It is easy to see that $|B_x(t)|\le |B_0(x)|$ for all $t\in\Real$
and $x\in(0,1)$. Moreover, the function $B_0$ is bounded on~$\Real$. Recall
that the measure~$\mu$ is also bounded. Thus our claim~\eqref{EQ_limImH}
follows form~\eqref{EQ_ReprImH} with the help of Lebesgue's Dominated
Convergence Theorem. This completes Step 1 of the proof.

\step2{Proof of inequality~\eqref{EQ_Re-h-to_-infty}.} This inequality is equivalent to
\begin{equation}\label{EQ_Re-H-to_-infty}
\limsup_{\UH_1\ni z\to 0}\Re H(z)<+\infty.
\end{equation}
Since $(\partial/\partial x)\Re H(x+iy)=\Re H'(x+iy)>0$ for all
$z:=x+iy\in\UH_1$, we have $\Re H(x+iy)<\Re H(1+iy)$ for any $y\in\Real$ and
$x\in(0,1)$. Passing in this inequality to the upper limit as $z=x+iy\to0$ we
obtain~\eqref{EQ_Re-H-to_-infty}. The proof is now complete.
\end{proof}

\begin{proof}[{\bfit Proof of Theorem~\ref{TH_anglims}}]
First of all, using Proposition~\ref{PR_unif-cont-inverse} together with the forerunning comment, we may assume
that the DW-point of~$(\phi_t)$ is $\tau=1$.

Fix any $\sigma\in\UC$. A fundamental family of Stolz angles with vertices
at~$\sigma$ is given by
$$
S_{\sigma,\alpha}:=\Big\{z\in\UD:\big|\arg(1-\overline\sigma
z)\big|<\alpha,\,\big|1-\overline\sigma
z\big|<(\cos\alpha)/2\Big\},\quad\alpha\in(0,\pi/2).
$$
Fix any Stolz angle $S$ in this family. Then $S$ is a domain in $\UD$,
with ${\partial S\cap\partial \UD=\{\sigma\}}$. That is why it follows from
Proposition~\ref{PR_Re_h} that $h|_S$ admits a continuous extension
to~$\closure S$ (as a mapping into~$\ComplexE$). Taking into account that $S$
is convex, it follows that $h|_S$ is uniformly continuous as a mapping from $S$
endowed with the Euclidean distance into the domain $\Omega:=h(\UD)$ endowed
with the distance~$d_\Omega$\,\footnote{Given two points $z_1,z_2\in S$, consider $h([z_1,z_2])$ as a candidate for $\Gamma$ in~\eqref{EQ_dist}.}, which has been introduced in
section~\ref{SS_PR_cont}.

Fix now $T>0$. The family of translations $\mathcal T:=(\ComplexE\ni w\mapsto
w+t)_{t\in[0,T]}$, where as usually we set $\infty+t=\infty$, is uniformly
equicontinuous in $\ComplexE$ w.r.t. the spherical distance. Since $\Omega$ is
invariant w.r.t. elements of~$\mathcal T$, it is easy to see that~$\mathcal T$
is a uniformly equicontinuous family of self-maps of~$\Omega$ endowed with the
distance~$d_\Omega$.

Finally, recall that $h$ is univalent. By
Proposition~\eqref{PR_unif-cont-inverse}, $h^{-1}$ is uniformly continuous
in~$\Omega$ w.r.t. the distance~$d_\Omega$.

Now combining the above facts, it is easy to conclude that the composite family
$$
\Big(S\ni z\mapsto \phi_t(z)=h^{-1}\big(h(z)+t\big)\in \UD\Big)_{t\in[0,T]}
$$
is uniformly equicontinuous in~$S$. Hence it admits uniformly equicontinuous
extension to~$\closure S$. The statement of Theorem~\ref{TH_anglims} follows
now immediately.
\end{proof}
\begin{proof}[{\bfit Proof of Proposition~\ref{PR_cont-in-t}}]
Let $\tau$ be the DW-point of~$(\phi_t)$. Clearly, it is sufficient to consider
cases~$\tau=0$ and $\tau=1$.

\case{1}{$\tau=1$.} In this case the K\oe{}nigs function~$h:\UD\to\Complex$
of~$(\phi_t)$ satisfies the Abel functional equation $h(\phi_t(z))=h(z)+t$ for
all $t\ge0$ and all $z\in\UD$. It follows that
$$
d_\Omega\big(h(\phi_{t_1}(z)),h(\phi_{t_2}(z))\big)\le|t_1-t_2|,\quad
\Omega:=h(\UD),
$$
for any $z\in\UD$ and any $t_1,t_2\in[0,+\infty)$.

Recall that~$h$ is univalent. Therefore, the statement of the proposition for
$z$ ranging in~$\UD$ follows from Proposition~\ref{PR_unif-cont-inverse}. It is
also true for the closed unit disk, because $\phi_t(z)$ is continuous in~$z$ on
each radius $[0,\sigma]$, where~$\sigma\in\UC$. Thus for $\tau=1$ the proof is
finished.

\case{2}{$\tau=0$.} Fix any $t_1\ge0$ and $t_2\ge t_1$. Using the Maximum
Principle for holomorphic functions we see that
$$\sup_{z\in\UD}|\phi_{t_2}(z)-\phi_{t_1}(z)|~~~\le~~~\sup_{z\in\UD}|\phi_{t_2-t_1}(z)-z|~~~=
\sup_{2/3<|z|<1}|\phi_{t_2-t_1}(z)-z|~~~=:~~\Upsilon(t_2-t_1).$$ It is enough
to show that $\Upsilon(t)\to0$ as $t\to+0$.

Taking advantage of Proposition~\ref{PR_lifting}, consider the one-parameter
semigroup formed by the functions $\psi_t:=p_0^{-1}\circ\tilde\phi_t\circ p_0$,
${t\ge0}$, where $p_0(\zeta):=(1+\zeta)/(1-\zeta)$ is the Cayley map. Then
$\phi_t(e^{-p_0(\zeta)})=e^{-p_0(\psi_t(\zeta))}$ for all $z\in\UD$. Hence
$$\Upsilon(t)=\sup_{\zeta\in\Pi}|e^{-p_0(\psi_t(\zeta))}-e^{-p_0(\zeta)}|,~~
\text{where }\Pi:=p_0^{-1}\big(\{\tilde x+i\tilde y:~0<\tilde x\le\log
3/2,\,|\tilde y|\le\pi\}\big).$$ By Case~1, for all $\zeta\in\Pi$,
$$
\psi_t(\zeta)\in\tilde \Pi:=p_0^{-1}\big(\{\tilde x+i\tilde y:~0<\tilde
x\le\log 2,\,|\tilde y|\le2\pi\}\big)
$$
provided $t>0$ is small enough. Since the derivative of the map $\zeta\mapsto
e^{-p_0(z)}$ is bounded on the convex hull of~$\tilde\Pi$, there exists $M>0$
such that $|e^{-p_0(\psi_t(\zeta))}-e^{-p_0(\zeta)}|\le M
|\psi_t(\zeta)-\zeta|$ for all $\zeta\in\Pi$ and  all~$t>0$ small enough.
Now by Case~1, it follows that $\Upsilon(t)\to0$ as~$t\to+0.$ The proof is now
complete.
\end{proof}

\section{Unrestricted limits at boundary fixed points}
\subsection{Main Theorem}
Now we formulate the main result of this paper. In~\cite[Corollary~3]{CMP2004}
it was proved that elements of a one-parameter semigroup with the interior
DW-point can be extended continuously to the boundary fixed points. This
statement can also be proved for the case when the DW-point $\tau\in\partial
\UD$. For a repelling boundary fixed\footnote{By repelling boundary fixed
points we mean all boundary fixed points (regular or non-regular) except for
the DW-point.} point~$\sigma$ the reason why $\phi_t$ has the unrestricted
limit at~$\sigma$ is essentially the same in both cases: the K\oe{}nigs function
has the unrestricted limit at all repelling boundary fixed points. However, for
$\sigma=\tau$ this is not true any more. The K\oe{}nigs does not need to have the
unrestricted limit at the boundary DW-point~$\tau$, see, \textsl{e.g.}, Example~\ref{EX1}.
\begin{theorem}\label{TH_cont-at-fixed-points}
Let $(\phi_t)$ be a one-parameter semigroup in~$\UD$. For each $t\ge0$ and each
$\sigma\in\UC$ denote $\phi_t(\sigma):=\angle\lim_{z\to\sigma}\phi_t(z)$. Then
for any $T>0$ the family of mappings
$$
\Phi_T:=\Big(\UDc\ni z\mapsto \phi_t(z)\in\UDc\Big)_{t\in[0,T]}
$$
is equicontinuous at every boundary fixed point of~$(\phi_t)$. Moreover, if the
semigroup~$(\phi_t)$ is of hyperbolic type\footnote{A one-parameter
semigroup~$(\phi_t)$ is said to be of \textit{hyperbolic type} if its DW-point
$\tau\in\UC$ and $\phi_t'(\tau)>1$ for some (and hence for all)~$t>0$. See,
\textsl{e.g.}, \cite{ElinShoikhetBook} or~\cite{ElinShoikhet} for more
details.}, then the family
$$
\Phi:=\Big(\UDc\ni z\mapsto \phi_t(z)\in\UDc\Big)_{t\ge0}
$$
is equicontinuous at the DW-point~$\tau$ of~$(\phi_t)$.
\end{theorem}
The proof of this theorem is given in Section~\ref{SS_proof}.

\begin{remark}\label{RM_no-for-contract-points}
By the theorem above, the unrestricted limit $\lim_{\UD\ni
z\to\sigma}\phi_t(z)$ exists at any boundary fixed point~$\sigma\in\UC$
of~$(\phi_t)$. One might ask if the same holds \textit{for the contact points
of~$(\phi_t)$}. The answer is: ``not necessarily'', see
Example~\ref{EX_contact} in section~\ref{SS_rem-loc-beh}.
\end{remark}

\subsection{K\oe{}nigs function of a one-parameter semigroup with the boundary
DW-point}\label{SS_Koenigs}

In this subsection we prove some auxiliary statements concerning K\oe{}nigs
functions of one-parametric semigroups with the boundary DW-point, and obtain
from them some consequences  characterizing the semigroup itself. Throughout
the subsection \textit{we will assume that $(\phi_t)$ is a one-parameter
semigroup with the DW-point $\tau=1$}. By $h$ we denote its K\oe{}nigs function,
and let $\Omega:=h(\UD)$.

First of all we need to make some general remarks and recall some definitions.
\begin{definition}
By a {\it slit} (or an {\it end-cut})~$\gamma$ in a domain~$D\subset\ComplexE$
we mean the image of~$[0,1)$ under an injective continuous mapping
$\varphi:[0,1]\to\closure D$ with $\varphi([0,1))\subset D$ and
$\varphi(1)\in\partial D$. The point $\varphi(1)$ is said to the the {\it root}
or the {\it landing point} of the slit~$\gamma$.
\end{definition}
\begin{remark}\label{RM_a.p.correspondence}
Two slits $\gamma_1$, $\gamma_2$ in a domain $D$ are called {\it equivalent} if
they share the same root $\omega_0\in\partial D$ and any neighbourhood
of~$\omega_0$ contains a curve $\Gamma\subset D$ that joins $\gamma_1$
and~$\gamma_2$. It is known (see,\,\textsl{e.g.},\,\cite[Theorem~1 in
Chapter~2, \S3]{Goluzin}) that if $D=f(\UD)$, where $f:\UD\to\ComplexE$ is a
conformal mapping, then the preimage $f^{-1}(\gamma)$ of any slit~$\gamma$
in~$D$ is  a slit in the unit disk~$\UD$. Moreover, in this case two slits
$\gamma_1$, $\gamma_2$ in~$D$ are equivalent if and only if their preimages
$f^{-1}(\gamma_1)$ and $f^{-1}(\gamma_2)$ land at the same point on~$\UC$. A
{\it cross-cut} in $D$ can be defined as a union of two non-equivalent slits
intersecting only at their common end-point in~$D$. It follows that the
preimage $f^{-1}(C)$ of any cross-cut~$C$ in $D$ is a cross-cut in~$\UD$.
\end{remark}

\begin{remark}\label{RM_slittoinfty}
Let $w_0\in\Omega$. Then $R_0:=\{w_0+t:t>0\}\subset\Omega$. Moreover, $R_0$ is
a slit in~$\Omega$, whose preimage $h^{-1}(R_0)$ is a slit in~$\UD$ landing at
the DW-point~$\tau=1$, because $h^{-1}(w_0+t)=\phi_t(h^{-1}(w_0))\to\tau$
as~$t\to+\infty$.
\end{remark}

First we prove that continuity of $h$ at a boundary point implies continuity of
$z\mapsto \phi_t(z)$ at that point locally uniformly w.r.t.~$t$.
\begin{proposition}\label{PR_if_h_cont_then_phi_cont}
Assume that $h$ has the unrestricted limit, finite or infinite, at a
point~$\sigma\in\UC$. Then the functions $\phi_t$ also have unrestricted limits
at~$\sigma$ and the convergence $\phi_t(z)\to\phi_t(\sigma)$ as $\UD\ni
z\to\sigma$ is locally uniform w.r.t.~$t\ge0$.
\end{proposition}
\begin{proof}
As an elementary argument of {\sl reductio ad absurdum} shows, it is sufficient
to prove that given any slit $\gamma$ in~$\UD$ landing at~$\sigma$, the
functions $\phi_t(z)$ tend to $\phi_t(\sigma)$ locally uniformly w.r.t.~$t$ as
$z$ tends to~$\sigma$ along~$\gamma$.

We use essentially the same idea as in the proof of~Theorem~\ref{TH_anglims}.
The restriction $h|_\gamma$ as a mapping from $\gamma$ endowed with the
Euclidean distance to the domain~$\Omega$ endowed with the distance~$d_\Omega$,
is uniformly continuous. For each $T>0$ the family of self-maps $(\Omega\ni
w\mapsto w+t\in\Omega)_{t\in[0,T]}$, is uniformly equicontinuous in~$\Omega$
w.r.t. the distance~$d_\Omega$. Using Proposition~\ref{PR_unif-cont-inverse}
for $f:=h$, we conclude that  for any $T>0$ the family $$\Big(\gamma\ni
z\mapsto \phi_t(z)=h^{-1}(h(z)+t)\Big)_{t\in[0,T]}$$ is uniformly
equicontinuous. This means that $\phi_t|_\gamma$ has the limit at~$\sigma$
locally uniformly in~$t\ge0$. According to Lemma~9.3 and Theorem~9.3 in
\cite[p.\,262--268]{Pommerenke}, all the asymptotic values of a univalent
holomorphic function in~$\UD$ at a given point of~$\partial \UD$, if any
exists, coincide with the angular limit at this point. Hence, $\lim_{\gamma\ni
z\to\sigma}\phi_t(z)=\phi_t(\sigma)$. The proof is now complete.
\end{proof}

Now we concentrate on the study of $h$ near the boundary fixed points
of~$(\phi_t)$. Note that the angular limit of $h$, which exists according to
Proposition~\ref{PR_Re_h}, equals $\infty$ at each boundary fixed point. Our
arguments use extensively the theory of boundary correspondence under conformal
mappings of simply connected domains. We refer the reader
to~\cite[Chapter~9]{ClusterSets} or \cite[\S\S2.4--2.5]{Pommerenke-II} for the
basic theory and definitions.

By $\partial E$ we will denote the boundary of a set $E\subset\ComplexE$. If
$E\subset \Complex$, we will write $\partial_\Complex E$ for $\partial
E\setminus\{\infty\}$. For a prime end~$P$ of a simply connected domain
$\Omega\subset\ComplexE$, we will denote by $I(P)$ its impression. A prime end
is said to be \textit{trivial} if its impression is a singleton.

Let $(C_n)$ be a null-chain of cross-cuts in~$\Omega$ representing some prime
end~$P$. For each $n\in\Natural$ denote by $D_n$ the connected component of
$\Omega\setminus C_n$ that contains $C_{n+1}$. We recall now one definition
from the theory of prime ends.
\begin{definition}\label{D_conv_prime-end}
A sequence $(w_k)\subset\Omega$ is said to {\it converge to the prime end~$P$},
if for the null-chain~$(C_n)$ representing the prime end~$P$ (and hence for all
such null-chains) the following statement holds: for every $n\in\Natural$,
there exists $k_0$ such that $w_k\in D_n$ whenever~${k>k_0}$. In a similar way
one defines convergence to a prime end for slits in~$\Omega$ and more general
continuous families $(w_x\in\Omega)_{x\in J}$, $J\subset\Real$.
\end{definition}
\begin{definition}\label{D_conv-null-chain}
We say that $(C_n)$ {\it converges to a point $w_0\in\partial\Omega$} if for
any neighbourhood~$\mathcal O$ of~$w_0$ all but a finite number of~$C_n$'s lie
in $\mathcal O$. Taking into account that $\diam_{\ComplexE}(C_n)\to0$ as
$n\to+\infty$ by the very definition of a null-chain, the equivalent condition
is that there exists a sequence $(w_n)\subset\Omega$ converging to~$w_0$ such that $w_n\in
C_n$ for all~$n\in\Natural$.
\end{definition}

Now we are going to prove that $h$ is continuous at every repelling boundary
fixed point of~$(\phi_t)$. The proof is based on the following lemma.
\begin{lemma}\label{LM_cross-cut-mimus-infty}
Let $x_0\in\Real$ and let $y_0:(-\infty,x_0]\to\Real$ be a continuous function.
Suppose that the graph $\Gamma:=\{x+i y_0(x):\,x\le x_0\}$ lies entirely
in~$\Omega:=h(\UD)$ and that there exist $w_1,w_2\in\Complex\setminus\Omega$
such that $\Re w_1,\Re w_2
> x_0$ and $\Im w_1<y_0(x_0)<\Im w_2$. Then $x\mapsto x+iy_0(x)$ converges,  as
$x\to-\infty$, to a trivial prime end of~$\Omega$.
\end{lemma}
\begin{proof}
For each $x\le x_0$, define $Y(x)$ to be the connected component of
$\{y\in\Real:x+iy\in\Omega\}$ that contains the point~$y_0(x)$. Since $\Omega$
is invariant under translations~$w\mapsto w+t$, $t>0$, it follows that $Y(x)$
is bounded for all $x\le x_0$ and that $Y(x')\subset Y(x)$ whenever $x'\le x\le
x_0$. (To check the latter statement one has to take into account that $y_0$ is
continuous.) We claim that the intervals
$$C_n:=\{x+iy:x=x_n,\,y\in Y(x_n)\},\quad\text{where }x_n:=x_0-n,$$
form a null-chain in $\Omega$. Indeed, each $C_n$ is a cross-cut in~$\Omega$,
the closures of $C_n$'s are pairwise disjoint, $\diam_{\ComplexE}(C_n)\to0$ as
$n\to+\infty$. Moreover, recall that since $C_n$ is a cross-cut,
$\Omega\setminus C_n$ has exactly two components for each $n\in\Natural$.
Consider the set
$$G_n:=\{x+iy:\,x<x_n,\,y\in Y(x_n) \}.$$ It is open and $\partial
G_n\cap(\Omega\setminus C_n)=\emptyset$. Therefore, $\Omega\setminus
C_n=D_n\cup D_n'$, where $D_n:=\big(\Omega\setminus C_n\big)\cap G_n=\Omega\cap
G_n$ and $D'_n:=\Omega\setminus \closure{G_n}$ are both open and nonempty.
(Indeed, note that $\Gamma\cap G_n\neq\emptyset$, but $\Gamma\not\subset
\closure{G_n}$.) Hence, $D_n$ and $D'_n$ are the two connected components of
$\Omega\setminus C_n$. Clearly, $C_{n+1}\subset D_n$ and $C_{n-1}\subset D'_n$.
Thus $(C_n)$ is a null-chain and the impression of the prime end~$P$ defined by
this null-chain is~$I(P)=\bigcap_n \closure{D_n}\subset \bigcap_n
\closure{G_n}=\{\infty\}$.

Finally by construction,  $D_n$ contains $x+iy_0(x)$ for all $x<x_n$.
Therefore, $x\mapsto x+iy_0(x)$ converges to the prime end $P$ as
$x\to-\infty$. This completes the proof.
\end{proof}

\begin{corollary}\label{C_h-cont-at-rep}
At every boundary repelling\footnote{By a boundary repelling fixed point we
mean a boundary fixed point~$\sigma$ such that
$\phi_t'(\sigma)\in(1,+\infty)\cup\{\infty\}$ for all~$t>0$, \textsl{i.e.}, a
boundary fixed point other than the DW-point of~$(\phi_t)$.}
 fixed point~$\sigma$ of the semigroup~$(\phi_t)$ the function~$h$ has
unrestricted limit~$\infty$, with $\Re h(z)\to-\infty$ as $z\to\sigma$.
\end{corollary}
\begin{proof}
Using M\"obus transformations of~$\UD$ fixing $\tau=1$, we may assume
that~$\sigma=-1$. Consider the function $S(r):=\Re h(-r)$, $r\in[0,1)$. Using
Remark~\ref{RM_H}, one can easily show that $S$ is monotonically decreasing.
Moreover, by~\eqref{EQ_limImH} in the proof of Proposition~\ref{PR_Re_h}, $\Im
h(-r)$ has a finite limit as~$r\to1-0$. At the same time, as we mentioned
in~Remark~\ref{RM_h-to-infty=fixed}, $h(-r)\to\infty$ as~$r\to1-0$. Taking into
account that $S$ is decreasing, we conclude that $S(r)\to-\infty$ as~$r\to1-0$
and consequently $J:=S\big([0,1)\big)=(-\infty,S(0)]$. Therefore, the curve
$\Gamma:=h\big((-1,0]\big)\subset\Omega$ is the graph of the function
$$J\ni x\mapsto y_0(x):=\Im h\big(-S^{-1}(x)\big).$$

Note that there exist $w_1,w_2\in\Complex\setminus\Omega$ with $\Im
w_1<y_0(x_0)<\Im w_2$ for some $x_0\le S(0)$. Otherwise, $\Gamma$, as a slit
in~$\Omega$, would be equivalent to $R_0:=\{t+h(0):t\ge0\}$. But the landing
point of $h^{-1}(R_0)$ is the DW-point~$\tau=1$, see
Remark~\ref{RM_slittoinfty}, while the landing point of $h^{-1}(\Gamma)$
is~$\sigma=-1$ by construction.

Thus we can apply
Lemma~\ref{LM_cross-cut-mimus-infty} to conclude that $h(-r)$ converges, as
$[0,1)\ni r\to1$, to a trivial prime end~$P$ of~$\Omega$. This means
(see,\,\textsl{e.g.},\,\cite[Chapter~9,\,\S4]{ClusterSets}) that the point
$\sigma=-1$ corresponds under~$h$ to this trivial prime end, and hence the
limit set of~$h(z)$ as $\UD\ni z\to-1$ is the singleton~$I(P)=\{\infty\}$.
Since for the null-chain~$(C_n)$ defining the prime end~$P$ that we have
constructed in the proof of~Lemma~\ref{LM_cross-cut-mimus-infty} one has
$\sup_{w\in D_n}\Re w\to-\infty$ as $n\to+\infty$, we may conclude that the
unrestricted limit of $\Re h$ at~$\sigma$ also exists and equals~$-\infty$.
\end{proof}

\begin{lemma}\label{LM_prime-end-nonDW}
Let $P$ be a non-trivial prime end of $\Omega:=h(\UD)$. Then $P$ corresponds
under the mapping~$h$ to the DW-point~$\tau$ if and only if
\begin{equation}\label{EQ_Re_unb}
\sup\{\Re w:\,w\in I(P)\cap\Complex\}=+\infty.
\end{equation}
\end{lemma}
\begin{proof}
Denote by $\sigma$ the unique point on~$\UC$ that corresponds to the prime
end~$P$ under the mapping~$h$.

The fact that if~\eqref{EQ_Re_unb} holds, then $\sigma=\tau$, follows readily
from Proposition~\ref{PR_Re_h}, because $I(P)$ is the limit set of~$h(z)$ as
$\UD\ni z\to\sigma$, see \textsl{e.g.} \cite[Theorem 9.4,
p.\,173]{ClusterSets}.

Now we prove the converse statement. So let us assume that $\sigma=\tau$. We
have to show that~\eqref{EQ_Re_unb} takes place. Choose any point $w_0\in
I(P)\cap\Complex$. Then there exists a sequence $(z_n)\subset\UD$ converging
to~$\tau$ such that $w_n:=h(z_n)$ converges to~$w_0$. Consider another sequence
$(\zeta_n)\subset\UD$ converging to~$\tau$ defined by $\zeta_n:=\phi_n(0)$ for
all~$n\in\Natural$. Then the segments $[z_n,\zeta_n]$ also converge to~$\tau$
and hence the limit set of the
sequence~$(\Gamma_n):=\big(h([z_n,\zeta_n])\big)$ is a subset of $I(P)$. On the
one hand, $w_n\in\Gamma_n$ for all $n\in\Natural$ and tends to~$w_0$ as
$n\to+\infty$. On the other hand, $h(0)+n=h(\zeta_n)\in\Gamma_n$ for all
$n\in\Natural$. It follows that for any $x>\Re w_0$ there exists a
sequence~$\omega_n\in\Gamma_n$ such that $\Re \omega_n\to x$. Let $\xi_x$ be
any limit point of~$(\omega_n)$. Then $\xi_x\in I(P)$ and $\Re \xi_x=x$. This
implies~\eqref{EQ_Re_unb} and thus the proof is complete.
\end{proof}

\begin{lemma}\label{LM_absorbing}
The domain $\Omega$ is absorbing for the domain $$\mathcal
D:=\{w\in\Complex:\exists\, \omega\in\Omega~\text{\rm such that}~\Im
(w-\omega)=0\}$$ w.r.t. the action of the translation semigroup ${(w\mapsto
w+t)_{t\ge0}}\subset\mathrm{Aut}(\mathcal D)$, which means that for any
$K\subset\subset\mathcal D$ there exists $t\ge0$ such that $K+t\subset\Omega$.
\end{lemma}
\begin{proof}
For $x\in\Real$, denote $J_x:=\{y\in\Real:x+iy\in\Omega\}$ and let
$J:=\bigcup_{x\in\Real} J_x$. In particular, $J$ is an open set in~$\Real$.
Note that, in fact, $J=\{\Im \omega:\,\omega\in\Omega\}$ and hence $J$ is
connected. By construction, $\mathcal D=\{w\in\Complex:\Im w\in J\}$. In
particular, $\mathcal D$ is a domain. If now $K$ is a compact subset
of~$\mathcal D$, then $K_1:=\{\Im w:\,w\in K\}$ is a compact subset of $J$ and
the family $(J_x)$ is an open cover for~$K_1$. Moreover, by the invariance of
$\Omega$ under translations $w\mapsto w+t$, $t\ge0$, we have $J_{x_1}\subset
J_{x_2}$ whenever $x_1\le x_2$. Hence there exists $x_0\in\Real$ such that
$K_1\subset J_{x_0}$. Moreover, since $K$ is compact, $\inf_{w\in K} \Re
w=:x_*\in\Real$. It follows that $K+(x_0-x_*)\subset\Omega$. The proof is now
complete.
\end{proof}
\begin{remark}\label{RM_mathcalD}
Note that the domain $\mathcal D$ defined above is either a horizontal strip,
or a half-plane whose boundary is parallel to~$\Real$, or the whole
plane~$\Complex$. Note that in the first case, $(\phi_t)$ is of hyperbolic
type, while in the second and the third cases it is of parabolic type,
see~\cite[p.\,256--257]{AnaFlows}.
\end{remark}

\begin{proposition}\label{PR_impressions}The following statements hold.
\begin{mylist}
\item[(A)]
Let $P$ be any prime end of $\Omega:=h(\UD)$. Then $I(P)\cap\Complex$ is
contained in the union of two straight lines parallel to the real axis.
\item[(B)]
 If in
addition,
\begin{equation}\label{EQ_Re-bounded}
\sup\{\Re w:\,w\in I(P)\cap\Complex\}<+\infty,
\end{equation}
then $I(P)\cap\Complex$ is contained on one straight line parallel to the real
axis.
\end{mylist}
\end{proposition}
\begin{proof}
Suppose on the contrary to~(A) that there exist $w_j\in I(P)\cap\Complex$,
$j=1,2,3$, such that $\Im w_1<\Im w_2<\Im w_3$. Denote by $\sigma$ the unique
point on~$\UC$ that corresponds to the prime end~$P$ under the mapping~$h$. We
proceed by constructing two slits in~$\Omega$ in the following way.

Fix any $y_1\in(\Im w_1, \Im w_2)$. The line $L_{y_1}:=\{w\in\Complex:\Im
w=y_1\}$ intersects $\Omega$, because $\Omega$ is connected and
$w_1,w_2\in\partial_\Complex\Omega$. Recall also that $\Omega$ is invariant
w.r.t. the translations $w\mapsto w+t$, $t>0$. It follows that, either
$L_y\subset\Omega$, and in this case we set $w_{y_1}:=\infty$,
$L^+_{y_1}:=L_{y_1}$, or there exists a point $w_{y_1}\in
L_{y_1}\cap\partial\Omega$ such that $$L_{y_1}^+:=\{w\in L_{y_1}:\Re w>\Re
w_{y_1}\}\subset\Omega\quad \text{and}\quad L_{y_1}^-:=L_{y_1}\setminus
L_{y_1}^+\subset\Complex\setminus\Omega.$$

Now fix $y_2\in(\Im w_2, \Im w_3)$ and construct in the same way $w_{y_2}$ and
$L_{y_2}$. By Lemma~\ref{LM_absorbing} there exists $x_0\in\Real$ such that the
segment $\Gamma_0:=[x_0+iy_1,x_0+iy_2]$ lies in~$\Omega$. Let
$\Gamma_j:=L_{y_j}^+\cap\{w:\Re w<x_0\}$, $j=1,2$. The union
$$\Gamma:=\bigcup_{j=0,1,2}\Gamma_j$$
is a cross-cut in~$\Omega$. Indeed, it is easy to see that the limits
$$\sigma_j:=\lim_{L_{y_j}^+\ni w\to w_j}h^{-1}(w),\quad j=1,2,$$ are different,
$\sigma_1\neq\sigma_2$, because the slits $\Gamma_1$ and $\Gamma_2$ are not
equivalent: otherwise we would have $w_{y_1}=w_{y_2}=\infty$ and
$\{w\in\Complex:\Im w\in[y_1,y_2]\}\subset\Omega$, which contradicts the
construction.

Moreover, we may assume $\sigma\not\in\{\sigma_1,\sigma_2\}$. Indeed, fix $j=1$ or $j=2$. If
$w_{y_j}=\infty$, then by Lemma~\ref{LM_cross-cut-mimus-infty}, the slit $\Gamma_j$ converges
to a trivial prime end, but $I(P)$ by the hypothesis contains more than one point. Hence in
this case $\sigma_j\neq\sigma$. If $w_{y_j}\neq\infty$ we {\sl a priori} may have the
situation that $\sigma_j=\sigma$. This would mean that the slit $\Gamma_j$ converges to $P$. If this happens choose in the above argument another value of $y_j\in(\Im w_{j},\Im w_{j+1})$.
Then the landing point $w_{y_j}$  will change and, since no two slits with different landing
points can converge to the same prime end~\cite[Theorem 9.7, p.\,177]{ClusterSets}, we now
meet the desired condition $\sigma_j\neq \sigma$.

Using the argument from the proof of Lemma~\ref{LM_cross-cut-mimus-infty}, one
can conclude that the connected components of $\Omega\setminus\Gamma$ are
$\Omega_1:=\Omega\setminus\closure G$ and $\Omega_2:=\Omega\cap G$, where
$$G:=\{w\in\Complex:\,\Re w<x_0,\,y_1<\Im w<y_2\}.$$ Clearly,
$w_1\in \partial\Omega_1\setminus\partial\Omega_2$ and
$w_2\in\partial\Omega_2\setminus\partial\Omega_1$. Recall that both $w_1$ and
$w_2$ belong to~$I(P)$, \textsl{i.e}., to the limit set of $h(z)$ as~$\UD\ni
z\to\sigma$. Therefore, $\sigma\in\partial h^{-1}(\Omega_1)\cap\partial
h^{-1}(\Omega_2)$. The latter means that $\sigma\in\{\sigma_1,\sigma_2\}$. This
contradicts the construction and thus proves part~(A) of the proposition.

To prove~(B) we use similar ideas. Assume on the contrary that the statement
does not hold. Then there exist ${w_1,w_2\in I(P)}$ with $\Im w_1<\Im w_2$. Fix
any $y_1\in(\Im w_1,\Im w_2)$ and let $L_{y_1}^+$ be constructed as above. Its
preimage $h^{-1}(L^+_{y_1})$ lands on the unit circle at two points, the
DW-point~$\tau$ and another point~$\varsigma\in\UC\setminus\{\tau\}$. As above,
we may assume that $\sigma\neq\varsigma$. Moreover, by
Lemma~\eqref{LM_prime-end-nonDW}, $\sigma\neq\tau$. This leads to a
contradiction in a similar way as in the proof of~(A) and thus shows that
statement (B) is also true.
\end{proof}

\begin{remark}
In fact we can also prove that if the prime end $P$ corresponds to the DW-point
and $I(P)\cap\Complex$ is not contained in one line, then for both lines, $L_1$
and $L_2$ whose union contains~$E$, we have $\sup\{\Re w:\,w\in I(P)\cap
L_j\}=+\infty$.
\end{remark}

\begin{proof}[\bfit{Proof of Proposition~\ref{Pr_repelling_iff}}]We divide the
proof into several steps. \Step1{We first prove the equivalence of (i), (ii)
and (iii).} Trivially (iii)$\Longrightarrow$(ii). Moreover, by
Corollary~\ref{C_h-cont-at-rep}, (i) implies (iii). It remains to show that
(ii) implies~(i). As we have already mentioned, if $h(r\sigma)\to\infty$ as
$r\to1-0$, then $\sigma$ is a boundary fixed point. Indeed, in this case for
each $t\ge0$, $\Gamma:=h([0,\sigma))$ and $\Gamma+t$ are two equivalent slits
in~$\Omega$. Therefore, by \cite[Theorem~1 in Chapter~2, \S3]{Goluzin},
$h^{-1}(\Gamma+t)$ is a slit in~$\UD$ landing at~$\sigma$, \textsl{i.e.}
$\phi_t(\sigma)=\lim_{r\to1-0}h^{-1}\big(h(r\sigma)+t\big)=\sigma$. So we only
have to show that if additionally $\Re h(r\sigma)\to-\infty$ as $r\to1-0$, then
$\sigma\neq\tau$.

Suppose on the contrary that (ii) holds and that~$\sigma$ coincides with the
DW-point~$\tau$ of~$(\phi_t)$. According Remark~\ref{RM_H} applied
with~$\sigma_0:=-1$ substituted for~$\sigma$, this means that the function
$H\in\Hol(\UH_1,\Complex)$ defined in Remark~\ref{RM_H} satisfy
$H(x)\to-\infty$ as $\Real\ni x\to+\infty$. But this is not possible, because
$\Re H'>0$ in~$\UH_1$. Thus, $\sigma\neq\tau$.

\Step2{Let us now pass to the proof of the statement concerning the limits of
$\Im h$. Assume first that  $\sigma$ is not a boundary fixed point (and in
particular does not coincide with the DW-point~$\tau$). We will show that in
this case the unrestricted limit of~$\Im h$ at~$\sigma$ exists finitely.} If
the impression~$I(P)$ of the prime end $P$ that corresponds under the map~$h$
to the point~$\sigma$, does not contain the point~$\infty$, then from
Proposition~\ref{PR_impressions} and Lemma~\ref{LM_prime-end-nonDW} it follows
that $I$ is a closed interval on a straight line parallel to~$\mathbb R$ and
hence the unrestricted limit $\lim_{\UD\ni z\to\sigma}\Im h(z)$ exists
finitely.

Now let us consider the case when the impression $I(P)$ of the prime end~$P$
contains~$\infty$. Since by assumption, $\sigma$ is not a boundary fixed point,
the argument of Step~1 shows that the angular limit
$h(\sigma):=\anglim_{z\to\sigma} h(z)$ is finite and hence $P$ is not trivial,
\textsl{i.e.}, $I(P)\neq\{\infty\}$. Using again
Proposition~\ref{PR_impressions} and Lemma~\ref{LM_prime-end-nonDW}, we
conclude that the impression is of the form $I(P)=\{w_0-x:\,x\ge0\}$, where
$w_0$ is some point\footnote{In fact, we could show that $h(\sigma)=w_0$, but
for our purposes it is enough to know that~$h(\sigma)\neq\infty$.}
on~$\partial_\Complex\Omega$. By \cite[Theorem~9.8]{ClusterSets} there exists a
null-chain $(C_n)$ belonging to the prime end~$P$ and converging\footnote{See
Definition~\ref{D_conv-null-chain}.} to~$w_1:=h(\sigma)\in I(P)\cap\Complex$.
Moreover, by the proof of \cite[Theorem~9.3]{ClusterSets}, we may assume that
each $C_n$ is an arc of the circle~$\tilde C_n:=\{w:\,|w-w_1|=r_n\}$ for some
$r_n$ positive and going to~$0$ as $n\to+\infty$. Denote by $w_n^{j}$, $j=1,2$,
$\Im w_n^1\le \Im w_n^2$, the end-points of~$C_n$ and by $D_n$ the connected
component of $\Omega\setminus C_n$ that contains $C_{n+1}$. By invariance
of~$\Omega$ w.r.t. the translations~$w\mapsto w+t$, $t\ge0$, the rays
$C_n^j:=\{w^j_n-x:\,x\ge0\}$ do not intersect $\Omega$. Fix any point
$w_*\in\Omega$ with $\Re w_*>A:=\Re w_1+\max\{r_n:{n\in\Natural}\}$. Since
$D_{n+1}\subset D_n$ for all $n\in\Natural$ and $\cap_{n\in\Natural}
D_n=\emptyset$, dropping a finite number of $C_n$'s we may assume that
$w_*\not\in D_n$ for all $n\in\Natural$. Then we have $\Im w_n^1\neq\Im w_n^2$,
because otherwise $D_n$ would be a subset of the disk bounded by $\tilde C_n$,
which is not possible since $I(P)\subset\closure{D_n}$.

Therefore the set $\Gamma_n:=C_n\cup C_n^1\cup C_n^2\cup\{\infty\}$ is a Jordan
curve. Clearly, one of the two Jordan domains bounded by~$\Gamma_n$ contains
the half-plane $\{w:\,\Re w>A\}$. We denote by $G_n$ the other Jordan domain
of~$\Gamma_n$. It is easy to see that $D_n\subset G_n$ for all~$n\in\UD$,
because $w_*\not\in D_n$. It follows that $|\Im w-\Im w_1|<r_n$ for any $w\in
D_n$ and all $n\in\Natural$. Recall that there exists a fundamental
system~$(U_n)$ of neighbourhoods of~$\sigma$ such that $h(U_n\cap \UD)=D_n$ for
all $n\in\Natural$. Thus $\Im h(z)\to\Im w_1$ as $\UD\ni z\to\sigma$.

\newcommand{\mes}{\mathop{\mathsf{mes}}}
\Step3{Now we assume that~$\sigma$ is a boundary non-regular fixed point
of~$(\phi_t)$. Again we have to show that the unrestricted limit of~$\Im h$
at~$\sigma$ exists finitely.} We use the arguments from the proofs of
Lemma~\ref{LM_cross-cut-mimus-infty} and Corollary~\ref{C_h-cont-at-rep}. In
the notation introduced in the proof of Lemma~\ref{LM_cross-cut-mimus-infty},
it sufficient to show that $\mes_\Real Y(x)\to0$ as ${x\to-\infty}$, where
$\mes_\Real\cdot$ stands for the length measure on~$\Real$. Then it would
follow that $\Im h(z)$ tends to the unique point in the intersection
$\bigcap_{x<x_0} Y(x)$. Suppose on the contrary that $\mes_\Real Y(x)>a$ for
some constant $a>0$ and all $x<x_0$. Recall that $Y(x)\subset Y(x')$ if $x>x'$.
Therefore, there exists an non-empty interval $Y_0$ which is a subset of~$Y(x)$
for all~$x<x_0$. It follows that $h([0,\sigma))$ is contained in the horizontal
strip $\{w:\Im w\in Y_0\}$, which in its turn is contained in the
domain~$\Omega$. According to~\cite[Theorem~2.5]{AnaFlows} this means
that~$\sigma$ is a boundary \textit{regular} fixed point. The contradiction
obtained proves the claim of this step.

\Step4{Assume finally that $\sigma$ is a boundary regular fixed point. We are
going to show that the limit set of $\Im h$ at~$\sigma$ is not a singleton.}
Assume first that $\sigma$ is not the DW-point. Then by \cite[Lemma~1]{CMP2006}
there exists a non-empty interval $Y_0\subset\Real$ such that the strip
$V:=\{w:\,\Im w\in Y_0\}$ is contained in~$\Omega$ and for every $w\in V$, the
$h^{-1}(w-t)\to \sigma$ as $t\to+\infty$. It immediately follows that the limit
set\footnote{We have shown a bit more: $Y_0$ is contained, in fact, in the
\textit{non-tangent limit set}
$$\big\{{y\in[-\infty,+\infty]}:\,{\exists (z_n)\subset\UD}\text{ s.t. }
{z_n\to \sigma}\text{ non-tangentially and }{\Im h(z_n)\to y}\big\}.$$} of $\Im
h(z)$ as $\UD\ni z\to\sigma$ contains~$Y_0$.

The proof for the case of the DW-point $\sigma=\tau$ is very similar. Take any
$w\in\Omega$. Then $h^{-1}(w+t)=\phi_t\big(h^{-1}(w)\big)\to\tau$ as
$t\to+\infty$. This means that the limit set of $\Im h$ at $\tau$ coincides
with the closure of $\{\Im w:\, w\in\Omega\}$. The proof is now complete.
\end{proof}


\begin{proposition}\label{PR_continuous_h}
Suppose that $h$ has no continuous extension to the DW-point~$\tau$. Then there
exist a line $L$ parallel to the real axis and a ray $R\subset L$ with
$\sup_{w\in R}\Re w=+\infty$ such that $L\cap\Omega=\emptyset$ and
$R\subset\partial\Omega$. In particular, the domain~$\mathcal D$ defined in
Lemma~\ref{LM_absorbing} is either half-plane of a horizontal strip.
\end{proposition}
\begin{proof}
Let $P$ be the prime end of~$\Omega$ that corresponds under the mapping~$h$ to
the DW-point~$\tau=1$. Since by the hypothesis $h$ has no continuous extention
to~$\tau$, the prime end $P$ is not trivial. Then it follows from
Lemma~\ref{LM_prime-end-nonDW} and Proposition~\ref{PR_impressions}(A) that the
impression~$I(P)$ contains a ray $R$ parallel to the real axis with $\sup_{w\in
R}\Re w=+\infty$. In particular, $R\subset\partial\Omega$. Then by the
translational invariance of~$\Omega$, the straight line $L$ containing~$R$
cannot intersect~$\Omega$. The proof is complete.
\end{proof}

\begin{proposition}\label{PR_hyperbol}
The following two statements are equivalent:
\begin{mylist}
\item[(A)] any half-plane $H$ bounded by a line parallel to~$\Real$, has
non-empty intersection with~${\Complex\setminus\Omega}$;
\item[(B)] the family
$$ \Phi:=\Big(\UDc\ni z\mapsto \phi_t(z)\in\UDc\Big)_{t\ge0}
$$
is equicontinuous at the DW-point~$\tau=1$ of~$(\phi_t)$.
\end{mylist}
\end{proposition}
\begin{proof}
First of all we notice that if (A) fails to hold, \textsl{i.e.}, if there
exists a half-plane $H\subset\Omega$ whose boundary is parallel to~$\Real$,
then the family $\Phi$ is not equicontinuous at~$\tau$. Indeed, take any line
$L\subset H$ parallel to $\Real$. Then the rays $L^+:=\{w\in L:\Re w\ge0\}$,
$L^-:=\{w\in L:\Re w\le0\}$ are two equivalent slits in~$\Omega$ and, as it
follows from Remarks~\ref{RM_a.p.correspondence}~and~\ref{RM_slittoinfty},
their preimages~$h^{-1}(L^+),\,h^{-1}(L^-)$ land at the DW-point~$\tau=1$. Take
any sequence $(w_n)\subset L^-$ tending to~$\infty$. Write $t_n:=-\Re w_n$.
Then on the one hand $z_n:=h^{-1}(w_n)\to\tau$ as $n\to\infty$. However, on the
other hand, $\phi_{t_n}(z_n)=h^{-1}(w_n+t_n)$ is the same point in~$\UD$ for
all~$n\in\Natural$. This shows that the family~$\Phi$ is not equicontinuous
at~$\tau$.

It now remains to prove that (A) implies (B). The idea of the proof is as
follows. Choose a point in $w_0\in\Omega$. By Remark~\ref{RM_slittoinfty} the
preimage $h^{-1}(R_0)$ of the ray~$R_0:=\{w_0+t:t\ge0\}$ is a slit landing at
the DW-point~$\tau=1$. In other words, $R_0$ as a slit in~$\Omega$, converges
to the prime end~$P$ corresponding under the map~$h$ to the DW-point~$\tau$. We
will construct a null-chain $(C_n)$ that represents the prime end~$P$ and which
has the following property: for each $n\in\Natural$, the connected
component~$D_n$ of the set~$\Omega\setminus C_n$ that contains $C_{n+1}$ is
invariant w.r.t. the translations $w\mapsto w+t$,~$t\ge0$.

For each $n\in\Natural$ denote $x_n:=\Re w_0+n$, $w_n:=x_n+i\Im w_0\in R_0$,
and let $\tilde C_n$ stand for the connected component of the set
$\{w\in\Omega:\Re w=x_n\}$ that contains the point~$w_n$. The following four
cases exhaust all possibilities:
\begin{mylist}
\item[Case 1:] for each $n\in\Natural$ the set $\tilde C_n$ is bounded.

\item[Case 2:] there exists $n_0\in\Natural$ such that for all $n>n_0$, $$\inf\{\Im w:w\in\tilde
C_n\}\in\Real,\quad \sup\{\Im w:w\in\tilde C_n\}=+\infty.$$

\item[Case 3:] there exists $n_0\in\Natural$ such that for all $n>n_0$, $$\inf\{\Im w:w\in\tilde
C_n\}=-\infty,\quad \sup\{\Im w:w\in\tilde C_n\}\in\Real.$$

\item[Case 4:] there exists $n_0\in\Natural$ such that for all $n>n_0$, $$\inf\{\Im w:w\in\tilde
C_n\}=-\infty,\quad \sup\{\Im w:w\in\tilde C_n\}=+\infty.$$
\end{mylist}

In {\it Case 1}, the sets $\tilde C_n$ are of the form $\tilde
C_n=(x_n+iy'_n,x_n+iy''_n)$, where $-\infty<y'_n<y_n''<+\infty$, and we set
$C_n:=\tilde C_n$ for all $n\in\Natural$. Because of the translational
invariance of~$\Omega$, $(y_n',y_n'')\subset (y_{n+1}',y_{n+1}'')$ for
every~$n\in\Natural$.

In {\it Case 2}, taking if necessary $w_{n_0}$ instead of~$w_0$,  we may assume
that $n_0=0$. Condition~(A) implies that there exists a strictly increasing
unbounded sequence $(y_n'')\subset(\Im w_0,+\infty)$ such that for every
$n\in\Natural$ the line $$L''_n:=\{w\in\Complex:\Im w=y_n''\}$$ is not a subset
of~$\Omega$. Note that $x_n+iy_n''\in L''_n\cap\Omega$. Now set
$y_n':=\inf\{\Im w:w\in\tilde C_n\}$, $x_n'':=\inf\{\Re w:w\in
L''_n\cap\Omega\}$ and let
$$
C_n:=(x_n+iy_n',x_n+iy_n'']\cup[x_n+iy_n'',x_n''+iy_n'').
$$
Note that again we have $(y_n',y_n'')\subset (y_{n+1}',y_{n+1}'')$ for
every~$n\in\Natural$.

{\it  Case 3} can reduced to the previous case by considering
$\overline{\phi_t(\bar z)}$ instead of~$\phi_t(z)$, which leads to passing from
$h(z)$ to $\overline{h(\bar z)}$. So we may skip Case 3.

In {\it Case 4} we also will assume that $n_0=0$. Fix a strictly increasing
unbounded sequence $(y_n'')\subset(\Im w_0,+\infty)$ and a strictly decreasing
unbounded sequence $(y_n')\subset(-\infty, \Im w_0)$ such that for
every~${n\in\Natural}$ the lines
$$L'_n:=\{w\in\Complex:\Im w=y_n'\}\quad\text{ and }\quad
L''_n:=\{w\in\Complex:\Im w=y_n''\}$$ are not subsets of~$\Omega$. Note that
$x_n+iy_n'\in L'_n\cap\Omega$ and $x_n+iy_n''\in L''_n\cap\Omega$. Now set
$x_n':=\inf\{\Re w:w\in L'_n\cap\Omega\}$, $x_n'':=\inf\{\Re w:w\in
L''_n\cap\Omega\}$ and let
$$
C_n:=(x_n'+iy_n',x_n+iy_n']\cup[x_n+iy_n',x_n+iy_n'']\cup[x_n+iy_n'',x_n''+iy_n'').
$$

Clearly, in all the cases for each $n\in\Natural$, $C_n$ is a slit in~$\Omega$,
the closures $\closure{C_n}$ are pairwise disjoint, and $\diam_{\ComplexE}
C_n\to0$ as $n\to+\infty$. To prove that $(C_n)$ is a null-chain, it remains to
show that for every $n\ge2$, $C_{n-1}$ and $C_{n+1}$ are contained in two
different connected components of $\Omega\setminus C_{n}$. Arguing as in the
proof of Lemma~\ref{LM_cross-cut-mimus-infty}, one can easily conclude that the
connected components of $\Omega\setminus C_n$ are
$\Omega\setminus\closure{G_n}$ and $\Omega\cap G_n$, where
$$G_n:=\{w\in\Complex:\,\Re w<x_n,\,y'_n<\Im w<y''_n\}.$$
By construction $C_{n+1}\cap G_n=\emptyset$, while $C_{n-1}\subset G_n$ because
$(y'_{n-1},y''_{n-1})\subset(y'_n,y''_n)$.

Thus $(C_n)$ is a null-chain, and $D_n=\Omega\setminus\closure{G_n}$ for
all~$n\in\Natural$. Hence it can be seen easily from the construction that the
slit $R_0$ converges to the prime end~$P$ represented by~$(C_n)$. Moreover,
$\bigcap_{n\in\Natural}D_n=\emptyset$, see,
\textsl{e.g.},\,\cite[p.\,170\,--\,171]{ClusterSets}. Thus, discarding a finite
number of cross-cuts in~$(C_n)$, we may assume that $h(0)\not\in D_n$ for
all~$n\in\Natural$.

Now fix any $n\in\Natural$. Since $t+G_n\supset G_n$ for any~$t\ge0$, we see
that $D_n$ is invariant w.r.t. the translations~$w\mapsto w+t$, $t\ge0$. This
means that the set $U_n:=h^{-1}(D_n)$ is invariant w.r.t. the
semigroup~$(\phi_t)$. Note that $h^{-1}(C_k)$, $k\in\Natural$, are cross-cuts
in~$\UD$, whose closures $\Gamma_k:=\closure{h^{-1}(C_k)}$ are pairwise
disjoint, see Remark~\ref{RM_a.p.correspondence}. Note that $0\not\in U_n$ by
construction. Therefore, $U_n=W_n\cap \UD$, where $W_n$ is the bounded Jordan
domain with $\partial W_n$ formed by $\Gamma_n$ together with its
reflection~$\Gamma_n^*$ w.r.t.~$\UC$. Furthermore, $\diam_{\ComplexE} C_n\to0$
implies, according to Proposition~\ref{PR_unif-cont-inverse}, that
$\diam_\Complex \Gamma_n\to0$ as $n\to+\infty$. Hence $\diam_\Complex W_n\to0$
as $n\to+\infty$. Finally, note that
$$\closure{h^{-1}(R_0)}=\closure{\{\phi_t(h^{-1}(w_0)):t\ge0\}}$$ is a Jordan arc that
joins $z=0\not\in W_n$ with the DW-point~$\tau$ and which, by the construction,
intersects $\partial W_n$ exactly once and in a non-tangential way. It follows
that~$\tau\in W_n$ for all $n\in\Natural$. Thus $(W_n)_{n\in\Natural}$ is a
neighbourhood basis of the point~$\tau$ with the property that
$\phi_t(W_n\cap\UD)\subset W_n$ for all $t\ge0$ and all~$n\in\Natural$.  It
follows that~(B) holds and thus the proof is completed.
\end{proof}

\begin{proposition}\label{PR_phi_cont_DW}
For each $T>0$ the family
$$ \Phi_T:=\Big(\UDc\ni z\mapsto
\phi_t(z)\in\UDc\Big)_{t\in[0,T]}
$$
is equicontinuous at the DW-point~$\tau=1$ of~$(\phi_t)$.
\end{proposition}
\begin{proof}
Essentially we will use the same idea as for the proof of implication
(A)$\Rightarrow$(B) in Proposition~\ref{PR_hyperbol}.

In view of Propositions~\ref{PR_if_h_cont_then_phi_cont} and~\ref{PR_hyperbol},
we may assume that the prime end $P$ that corresponds under the mapping~$h$ to
the DW-point~$\tau$ is non-trivial and, employing also
Proposition~\ref{PR_continuous_h}, that there exist two half-planes $H_1,H_2$
with $\partial_\Complex H_j$ parallel to~$\Real$, $j=1,2$, such that
$H_1\subset\Omega\subset H_2$. Without loss of generality we may assume that
$H_j$'s are of the form
$$
H_j:=\{w\in\Complex:\Im w<y_j\}
$$
for some $y_1<y_2\in\Real$. Fix any $w_0\in H_1$.

Let $n\in\Natural$. Denote $C_n:=\tilde C_{-n}\cup\{w:\Im(w-w_0)\le0,\,
|w-w_0|=n\}\cup\tilde C_n$, where $\tilde C_m$, $m\in\mathbb Z$, stands for the
connected component of $\{w\in\Omega:\Re (w-w_0)=m,\,\Im (w-w_0)\ge0\}$ that
contains the point $w_0+m$. Clearly, $C_n$'s are cross-cuts in $\Omega$ with
pairwise disjoint closures. Moreover, $\diam_{\ComplexE} C_n\to0$ as
$n\to+\infty$.

For $m\in\mathbb Z$ denote $y'_m:=\sup\{\Im w:w\in\tilde C_m\}$. Note that
$y_m'<y_2<+\infty$ because $\Omega\subset H_2$. By the translational invariance
of~$\Omega$, we have $y_m'\le y_k'$ whenever $m<k$. It follows that for any
$n\ge2$, $C_{n+1}\subset \Omega\setminus G_n$ and $C_{n-1}\subset G_n$, where
\begin{multline*}
G_n:=\{w:\Im(w-w_0)<0,\, |w-w_0|<n\}\\\cup\{w:|\Re (w-w_0)|<n,\,\Im w_0\le\Im
w<y'_n\}\\\cup\{w:\Re (w-w_0)\le -n,\,y_{-n}'<\Im w<y_n'\}.
\end{multline*}
(The last set in the union may be empty.) Arguing as in the proof of
Lemma~\ref{LM_cross-cut-mimus-infty}, one can conclude that for each $n\ge 2$
the sets $D_n:=\Omega\setminus\closure{G_n}\supset C_{n+1}$ and
$D'_n:=\Omega\cap G_n\supset C_{n-1}$ are the two connected components
of~$\Omega\setminus C_n$. Thus $(C_n)$ is a null-chain.

We also notice that if $k>n$, then $\closure{G_n}\subset \closure{G_k}+t$ for
all $t\in[0,k-n]$. It follows that
\begin{equation}\label{EQ_vlozh}
D_k+t=(\Omega\setminus\closure{G_k})+t=(\Omega+t)\setminus(\closure
{G_k}+t)\subset \Omega\setminus(\closure{G_k}+t)\subset \Omega\setminus\closure
{G_n}=D_n
\end{equation}
for all $t\in[0,k-n]$.

Fix $T>0$. Then according to~\eqref{EQ_vlozh}, $\phi_t(U_{k(n)})\subset U_n$
for all $t\in[0,T]$ and all $n\in\Natural$, where $U_n:=h^{-1}(D_n)$,
$k(n):=n+[T]+1$, and $[\,\cdot\,]$ stands for the integer part of a real
number. Note that $w_0\not\in D_n$ for all $n\in\Natural$ and that
$R_0:=\{w_0+t:t\ge0\}$ intersect each of $C_n$'s exactly once and
non-tangentially. Arguing now as in the proof of Proposition~\ref{PR_hyperbol},
we conclude that the family $\Phi_T$ is equicontinuous at~$\tau$.
\end{proof}

\subsection{Proof of the Main Theorem (Theorem~\ref{TH_cont-at-fixed-points})}\label{SS_proof}
By using M\"obius transformations of the unit disk, one always may assume that
the DW-point of $(\phi_t)$ is either $\tau=0$ (interior DW-point) or $\tau=1$
(boundary DW-point). Furthermore, using Proposition~\ref{PR_lifting} one can
reduce the case of the interior DW-point to the case of the boundary DW-point.
Therefore, without loss of generality we will assume that~$(\phi_t)$ has the
DW-point at~$\tau=1$.

Let $\sigma$ be a repelling boundary fixed point of~$(\phi_t)$. Then by
Corollary~\ref{C_h-cont-at-rep}, $h$ has the unrestricted limit at~$\sigma$.
Therefore, by Proposition~\ref{PR_if_h_cont_then_phi_cont}, for every $T>0$ the
family $\Phi_T$ is equicontinuous at~$\sigma$.

Note that by Proposition~\ref{PR_phi_cont_DW}, the family $\Phi_T$ is also
equicontinuous at the DW-point~$\tau$.

It remains to notice that if $(\phi_t)$ is of hyperbolic type, which means that
the angular derivative $\phi_t'(\tau)>1$ for $t>0$, then by
\cite[Theorem~2.1]{AnaFlows} the domain~$\Omega$ is contained in a horizontal
strip. Consequently, in this case by Proposition~\ref{PR_hyperbol}, the
family~$\Phi$ is equicontinuous at~$\tau$. This completes the proof.\qed

\section{A few examples}\label{SS_rem-loc-beh}
\subsection{Local behaviour near the boundary DW-point} Recall that if $(\phi_t)$ is a
one-parameter semigroup in~$\UD$, then $\phi_t$ converges, as $t\to+\infty$, to
the DW-point~$\tau$ \textit{locally uniformly} in~$\UD$. In the case of a
hyperbolic semigroup~$(\phi_t)$, although the DW-point~$\tau$ belongs to~$\UC$,
according to Theorem~\ref{TH_cont-at-fixed-points} the family $\Phi$ is
equicontinuous at~$\tau$. It is interesting to notice that these two facts do
\textbf{not} imply that there exists a neigbourhood~$W$ of~$\tau$ such that
$\phi_t\to\tau$  \textit{uniformly} in $\UD\cap W$ as $t\to+\infty$. For
instance, there can be infinitely many boundary fixed points in any
neighbourhood of~$\tau$, as the following example shows.

\begin{example}\label{EX1}
Denote $S:=\{w:|\Im w|<1\}$, $I'_n:=\{x+i(1-1/n):x\le n\}$, and
$I_n'':=\{x-i(1-1/n):x\le n\}.$ Then
$$\Omega:=S\setminus\bigcup_{n=2}^{+\infty}(I_n'\cup I_n'')$$ is a domain
invariant w.r.t. the translations $w\mapsto w+t$, $t\ge0$. Denote by $h$ the
conformal mapping of $\UD$ onto $\Omega$ normalized by the conditions $h(0)=0$
and $\lim_{t\to+\infty} h^{-1}(t)=1$. Then the formula
$\phi_t(z):=h^{-1}(h(z)+t)$ defines a one-parameter semigroup with the DW-point
$\tau=1$ and $h$ is its K\oe{}nigs function.

By \cite[Theorem~2.1]{AnaFlows}, $(\phi_t)$ is of hyperbolic type. Thus $\Phi$
is equicontinuous at~$\tau$. However, $h$ has no unrestricted limit at~$\tau$,
because the impression $I(P)$ of the prime end~$P$ that corresponds to~$\tau$
under the map~$h$, is the whole boundary of~$S$. To see that any neighbourhood
of $\tau$ contains infinitely many boundary fixed points let us return to the
proof of implication (A)$\Rightarrow$(B) in Proposition~\ref{PR_hyperbol}. Take
$w_0:=0$. Then for each $n\ge2$, the domain $D_n$ constructed in the proof of
that proposition, contains the strip $S_n:=\{w:1/n<\Im w<1/(n+1)\}$. According
to \cite[Theorem~2.5]{AnaFlows}, there exists a repelling boundary fixed point
$\sigma_n$ such that for every $w\in S_n$, $h^{-1}(w-t)\to\sigma_n$ as
$t\to+\infty$. Hence $\sigma_n\in \closure{W_n}$. It remains to recall that the
sequence $(W_n)$ form a neighbourhood basis of the point~$\tau$.
\end{example}

The reason why in the above example the uniform convergence of $\phi_t\to\tau$
fails in $W\cap\UD$ for any neighbourhood~$W$ of~$\tau$ is the presence of
repelling boundary fixed points. In fact, the following statement holds.
\begin{remark}\label{RM_conv-to-DW}
Let $(\phi_t)$ be any one-parameter semigroup in~$\Hol(\UD,\UD)$ with the
DW-point ${\tau\in\partial\UD}$. Let $\sigma\in\UDc$. Then either
$\phi_t(\sigma)=\sigma$ for all $t\ge0$, or $\phi_t(\sigma)\to\tau$ as
$t\to+\infty$. (Recall that for the case $\sigma\in\UC$, $\phi_t(\sigma)$
stands for the angular limit of~$\phi_t$ at~$\sigma$). Indeed, for
$\sigma\in\UC$, this follows from the Denjoy\,--\,Wolff Theorem, see
Remark~\ref{RM_cont-DW-theorem}. So assume that $\sigma\in\UC$. Then by
Proposition~\ref{PR_Re_h}, $\Gamma:=h([0,\sigma))$, where $h$ is the K\oe{}nigs
function of~$(\phi_t)$, is a slit in the domain~$\Omega:=h(\UD)$. For $t\ge0$
denote $\Gamma_t:=h(\phi_t([0,\sigma)))=t+\Gamma$. Let us assume first that
$\Gamma$ is bounded (as a subset of~$\Complex$), \textsl{i.e.}, it lands at
some point of~$\partial\Omega\cap\Complex$. We claim that
$\sup\big\{|h^{-1}(w)-\tau|:w\in \Gamma_t\big\}\to0$ as $t\to+\infty$, which is
equivalent to~$\phi_t\to\tau$ uniformly on~$[0,\sigma)$ and hence implies that
$\phi_t(\sigma)\to\tau$ as $t\to+\infty$. Indeed, recall again that
$L_w:=\{w+x:x\ge0\}\subset\Omega$ for any $w\in\Omega$. By boundedness
of~$\Gamma$,
$$\sup_{w\in\Gamma_t}\diam_{\ComplexE}(L_w)=
\sup_{\substack{w\in\Gamma\\x\ge0}}\frac1{\sqrt{1+|w+t+x|^2}}\to0\quad\text{as
$t\to+\infty$}.$$ Since $h^{-1}(L_w)$ is a slit in~$\UD$ landing at the
DW-point~$\tau$, our claim follows now from
Proposition~\ref{PR_unif-cont-inverse}.

It remains to consider the case when $\Gamma$ is not bounded, \textsl{i.e.},
the case when $\Gamma$ lands at~$\infty$. In this case $\Gamma_t$ is also a
slit in~$\Omega$ landing at~$\infty$ for any~$t\ge0$. Moreover, $\Gamma$ and
$\Gamma_t$ are {\it two equivalent slits} in~$\Omega$ for any~$t\ge0$, because
$[w,w+t]\subset\Omega$ for any $w\in\Gamma$. This means that $h^{-1}(\Gamma_t)$
lands at the same point as~$h^{-1}(\Gamma)=[0,\sigma)$, \textsl{i.e.}, at the
point~$\sigma$. Thus $\phi_t(\sigma)=\sigma$ for~all~$t\ge0$.
\end{remark}
\begin{remark}
It might be interesting to compare the statement of the previous remark with
analogous results for \textit{discrete iteration} in~$\UD$, see,
\textsl{e.g.},~\cite{Poggi_Pointwise} and~\cite[Section~5]{parab_zoo},
asserting that under some additional conditions on~$\phi\in\Hol(\UD,\UD)$, the
orbits $(\phi^{\circ n}(\sigma))_{n\in\Natural}$ converge to the DW-point
of~$\phi$ for a.e.~$\sigma\in\UC$ (where $\phi(\sigma)$ stands again for the
angular limit at~$\sigma$, whenever it exists).
\end{remark}

In view of Remark~\ref{RM_conv-to-DW}, it might look to be a plausible
conjecture that if the family $\Phi$ is equicontinuous at the boundary
DW-point~$\tau$ and if there is a neighbourhood~$W$ of~$\tau$ that does not
contain other boundary fixed points, then $\phi_t\to\tau$ as $t\to+\infty$
uniformly in~$W\cap\UD$. However, the following example disproves this
conjecture.

\begin{example}
Consider the domain
\begin{multline*}
\tilde S:=S_0\setminus\bigcup_{n=2}^{+\infty}(J'_n\cup
J''_n),\text{~~~where~}S_0:=\left\{u+iv: |v|<1;\,
u>-\frac1{1-|v|}\right\},\\
J_n':=\left[-n+i\left(1-\frac1n\right),\,i\left(1-\frac1n\right)\right],
~~J_n'':=\left[-n-i\left(1-\frac1n\right),\,-i\left(1-\frac1n\right)\right].
\end{multline*}
The segments $J_n'$ and $J_n''$ are slits in~$S_0$ landing on the curve
$u=-1/(1-|v|)$, $|v|<1$. Clearly, there exists no continuous map $F:[0,1)\to
\tilde S$ such that ${\lim_{[0,1)\ni x\to1} \Re F(x)=-\infty}$.

Denote by $\tilde S(y_1,y_2)$, where $y_1<y_2$, the image of $\tilde S$ under
the affine map  $u+iv\mapsto u+i(av+b)$, $a:=(y_2-y_1)/2$, $b:=(y_1+y_2)/2$,
chosen in such a way that the minimal strip containing $\tilde S(y_1,y_2)$ is
$\{w\in\Complex:y_1<\Im w<y_2\}$. Now we consider the domain $\Omega$
constructed in Example~\ref{EX1} and ``fill in'' with $\tilde S(y_1,y_2)$'s,
for appropriately chosen parameters~$y_1,y_2$, each of strips one obtains by
removing from $\Omega$ all the straight lines containing the slits $I'_n$ and
$I''_n$, $n\ge2$. In a more strict language, we consider the domain
\begin{multline*}
\tilde\Omega:=\Omega\setminus K(-1/2,1/2)\setminus\\
\bigcup_{n=2}^{+\infty}\Big(K\big(1-1/n,1-1/(1+n)\big)\cup
K\big(-1+1/(n+1),-1+1/n\big)\Big),\\\text{~~~where~~}
K(y_1,y_2):=\{w\in\Complex:y_1<\Im w<y_2\}\setminus\tilde S(y_1,y_2).
\end{multline*}
The set $\tilde \Omega$ is a simply connected domain in~$\Complex$ containing
the origin and invariant w.r.t. the translations $w\mapsto w+t$. Therefore,
there exists a unique conformal mapping $h$ of~$\UD$ onto~$\tilde \Omega$ with
$h(0)=0$, $\lim_{t\to+\infty} h^{-1}(t)=1$, which is the K\oe{}nigs function of
the one-parameter semigroup $(\phi_t):=(h^{-1}\circ(h+t))$ with the
DW-point~$\tau:=1$. First of all we notice that this semigroup has no repelling
boundary fixed points. Indeed, if $\sigma\in\UC\setminus\{\tau\}$ is a boundary
fixed point, then $[0,1)\ni x\mapsto F(x):=h(\sigma x)$ is a continuous map
into~$\tilde \Omega$, and by Proposition~\ref{Pr_repelling_iff}, $\Re
F(x)\to-\infty$ as $x\to1-0$. However, it is easy to see from the definition
of~$\tilde \Omega$ that there exists no mapping with these properties.

Note also that, as in Example~\ref{EX1}, $(\phi_t)$ is of hyperbolic type and
thus the family $\Phi$ is equicontinuous at~$\tau=1$. It remains to see that
there exists no neighbourhood~$W$ of~$\tau$ such that $\phi_t$ converges
uniformly to $\tau$ on~$W\cap\UD$. To this end take~$w_0=0$ and define~$(D_n)$
and $(W_n)$ as in the proof of implication (A)$\Rightarrow$(B) in
Proposition~\ref{PR_hyperbol}. Since for each $n\in\Natural$, we have
${\inf\{\Re w:w\in D_n\setminus D_{n+1}\}=-\infty}$, there exists no~$t\ge0$
such that $D_n+t\subset D_{n+1}$. Therefore, there exists no~$t\ge0$ such that
$\phi_t(W_n\cap\UD)\subset W_{n+1}\cap\UD$. Recall that~$(W_n)$ is a
neighbourhood basis of~$\tau$. Thus, although $\phi_t(z)\to\tau$ as
$t\to+\infty$ for all $z\in\UDc$, this convergence is not uniform in any
neighbourhood of~$\tau$.
\end{example}

\subsection{Contact points}
In Remark~\ref{RM_no-for-contract-points} we mentioned that the
Theorem~\ref{TH_cont-at-fixed-points} cannot be extended to \textit{contact
points}. To demonstrate this fact, we now present an example of a one-parameter
semigroup~$(\phi_t)$ with a contact point at which there exists no unrestricted
limit of~$\phi_t$'s.

\begin{example}\label{EX_contact}
Consider the domain
$$\Omega:=\UH_i\setminus\bigcup_{n\in\Natural}\{x+i/n:x\le0\}.$$
Clearly, this domain is invariant w.r.t.
the right translations. Therefore with an appropriate choice of a conformal
mapping~$h$ of~$\UD$ onto~$\Omega$, we obtain a one-parameter semigroup
$\phi_t:=h^{-1}\circ(h+t)$, $t\ge0$, with the DW-point~$\tau=1$.
 Moreover, $\Omega$ has a unique prime~$P$ end whose
impression is~$(-\infty,0]$ and this prime end contains an accessible boundary
point, \textsl{i.e.}, there is a slit $\Gamma$ in~$\Omega$ that converges
to~$P$ (\textsl{e.g.}, we can take~$\Gamma:=(0,1+i]$). This slit lands at the
point~$w_0=0$. For $t>0$ the translate~$\Gamma_t:=\Gamma+t$ of~$\Gamma$ is also
a slit in~$\Omega$, with landing point at~$w=t\in\partial\Omega$. Note that $\{w:\Re w>0,\,\Im w>0\}\subset\Omega\subset\UH_i$ and hence $h^{-1}$ has a continuous injective extension
to $\Omega\cup(0,+\infty)$. It follows that for each $t>0$, $\phi_t$ has a
contact point at the preimage~$\sigma_0$ of the prime end~$P$ under~$h$, but
does not have the unrestricted limit at~$\sigma_0$. In this example, $\sigma_0$
is not a \text{regular} contact point, \textsl{i.e.},
$\phi'_t(\sigma_0)=\infty$ for all $t>0$. However, a simple modification of
this example (take, \textsl{e.g.}, the domain
$\Omega':=\Omega\cup\{w:|w-i|<1\}$ instead of~$\Omega$) shows that even if we
consider a regular contact point, there still do not need to exist the
unrestricted limit at that point.
\end{example}

\section{A remark on evolution families admitting continuous extension to the
boundary}\label{S_EF-diskalgebra} The notion of an evolution family in
$\Hol(\UD,\UD)$ is a natural extension for that of a one-parameter semigroup to
the non-autonomous setting. It goes back to the seminal paper~\cite{Loewner}
that gave rise to a theory which is now known as Loewner Theory and which has
proved to be a powerful tool in the Geometric Function Theory and its
applications, see, \textsl{e.g.}, the survey~\cite{MFMS_ems}. We use the
general definition of an evolution family introduced in~\cite{BCM1}, see
also~\cite{SemigroupsGor}.
\begin{definition}\label{D_evol_fam}
Let with $d\in [1,+\infty]$. A family $(\varphi_{s,t})_{0\leq s\leq
t<+\infty}\subset\Hol(\UD,\UD)$ is said to be an {\sl evolution family of order $d$ in the
unit disk~$\UD$} if it satisfies the following conditions:
\begin{enumerate}
\item[EF1.] $\varphi_{s,s}=\id_{\mathbb{D}},$ for all $s\ge0$,

\item[EF2.] $\varphi_{s,t}=\varphi_{u,t}\circ\varphi_{s,u}$ whenever $0\leq
s\leq u\leq t<+\infty,$

\item[EF3.] for all $z\in\mathbb{D}$ and for all $T>0$ there exists a
non-negative function $k_{z,T}\in L^{d}([0,T],\mathbb{R})$ such that
\[
|\varphi_{s,u}(z)-\varphi_{s,t}(z)|\leq\int_{u}^{t}k_{z,T}(\xi)d\xi
\]
whenever $0\leq s\leq u\leq t\leq T.$
\end{enumerate}
\end{definition}

Similar to one-parameter semigroups, any evolution family in~$\UD$ is formed
by solutions to initial value problems for a specific first order ODE, driven
by the so-called {\it Herglotz vector fields}. These non-autonomous vector
fields can be regarded as locally integrable families of infinitesimal
generators; in particular, one-parametric semigroups and their infinitesimal
generators are special cases of evolution families and their corresponding
Herglotz vector fields, see \cite{BCM1} for the details. Therefore, the class
of holomorphic mappings~$\varphi\in\Hol(\UD,\UD)$ that can be embedded into an
evolution family in~$\UD$, \textsl{i.e.}, the class of all~$\varphi$ such that
$\varphi_{s,t}=\varphi$ for some $s\ge0$, $t\ge s$ and some evolution
family~$(\varphi_{s,t})$, is much large than that for one-parametric semigroups
in~$\UD$. In fact, using~\cite[Lemma~2.8]{SMP} one can deduce from the
Parametric Representation of bounded normalized univalent functions, see,
\textsl{e.g.},~\cite[Theorem~5, p.\,70]{Aleksandrov}, that this class,
regardless the order of the evolution families to be considered, coincides with
the set of all univalent holomorphic self-maps of~$\UD$.

Therefore, the results of this paper on one-parameter semigroups cannot be
extended, in general, to evolution families. However, an analogue of
Proposition~\ref{PR_cont-in-t} holds under additional condition that
\begin{itemize}
 \item[CNT1.]  $(\varphi_{s,t})_{t\ge s\ge0}\subset\dAlg$, where $\dAlg$ stands
for the class of all holomorphic functions in~$\UD$ admitting continuous
extension to the closed unit disk~$\UDc$.
\end{itemize}
We endow $\dAlg$ with the Chebysh\'ov (supremum) norm
${\|\varphi\|_{\dAlg}:=\sup_{z\in\UD}|\varphi(z)|}$. Each element $\varphi$
of~$\dAlg$ is identified with its extension to~$\UDc$.

\begin{remark}
From~\cite[Proposition 3.5]{BCM1} it follows that if $(\varphi_{s,t})$ is an evolution family
in~$\UD$ of some order~$d\in[1,+\infty]$, then
\begin{itemize}
\item[CNT2.] for each $z\in\UD$ the mapping $(s,t)\mapsto\varphi_{s,t}(z)$ from
$\{(s,t)\in\Real^2:0\le s\le t\}$ to $\UD$ is separately continuous in $s$ and $t$.
\end{itemize}
Moreover, all elements of an evolution family are univalent functions.
\end{remark}

\begin{proposition}\label{PR_EF-diskalgebra}
If a family $(\varphi_{s,t})_{0\le s\le t<+\infty}$ satisfies conditions EF1, EF2, CNT1, and
CNT2 and none of the functions $\varphi_{s,t}$ is constant, then for each $s\ge0$ the mapping
$[s,+\infty)\ni t\mapsto \varphi_{s,t}\in\dAlg$ is continuous.
\end{proposition}
\begin{proof}
Let us fix $s\ge0$. Consider any convergent sequence $(t_n)\subset[s,+\infty)$
and denote by $t_0$ the limit of~$(t_n)$. We have to prove that
$\|\varphi_{s,t_n}-\varphi_{s,t_0}\|_{\dAlg}\to0$ as $n\to+\infty$. Owing to
the Arzel\`a\,--\,Ascoli Theorem, it follows from condition~CNT2 that we only
have to show that the sequence $(\varphi_{s,t_n})_{n\in\Natural}$ is
equicontinuous on~$\UDc$. Suppose it is not the case. Then one can find a
sequence of Jordan arcs $(\gamma_n)$ lying in $\UD$ such that
$\diam_\Complex(\gamma_n)\to0$ as $n\to+\infty$, but
$\diam_\Complex(\varphi_{s,t_n}(\gamma_n))>\varepsilon$ for all $n\in\Natural$
and some $\varepsilon>0$ not depending on $n$.

Choose $T>0$ such that $t_n<T$ for all $n\in\Natural$. Then
\begin{equation}\label{EF2c}
\varphi_{t_n,T}\circ\varphi_{s,t_n}=\varphi_{s,T},\quad n\in\Natural.
\end{equation}
Denote  $C_n:=\varphi_{s,t_n}(\gamma_n)$.  By~\eqref{EF2c}, we have
$\diam_\Complex(\varphi_{t_n,T}(C_n))=\diam_\Complex(\varphi_{s,T}(\gamma_n))$. Since
$\varphi_{s,T}\in\dAlg$, the latter quantity tends to~$0$ as $n\to+\infty$. At the same time
$\diam_\Complex(C_n)>\varepsilon$ for each $n\in\Natural$. By the Schwarz\,--\,Pick  theorem,
\begin{equation}\label{SchP}
|\varphi'_{t_n,T}(z)|\le \frac{1}{1-|z|^2},\quad z\in\UD, n\in\Natural.
\end{equation}
It follows now from~\cite[Theorem~9.2, p.\,265]{Pommerenke} that the sequence
$\varphi_{t_n,T}$ has a subsequence converging to a constant. This fact contradicts the
hypothesis and thus completes the proof.
\end{proof}

\begin{remark}
Note that the hypothesis of the above proposition does not imply the continuity
of $[0,t]\ni s\mapsto\varphi_{s,t}\in\dAlg$. Indeed, the well-known in Loewner
Theory example constructed by Kufarev~\cite{Kufarev} (see also
\cite[p.\,43]{Aleksandrov}) reveals an evolution family
$(\varphi_{s,t})\subset\dAlg$ that fails to be continuous in $s$ at $s:=0$
w.r.t. the norm $\|\cdot\|_{\dAlg}$. In this example, for fixed $t>0$ and for
each $s\in(0,t)$ the function $\varphi_{s,t}$ maps~$\UD$ onto $\UD$ minus the
slit along a part $\Gamma_{s,t}$ of a hyperbolic geodesic $\Gamma_t$, while the
mapping $\varphi_{0,t}$ maps $\UD$ onto the connected component of
$\UD\setminus\Gamma_t$ that contains the origin. Since
$\varphi_{s,t}(\UDc)=\UDc\neq\varphi_{0,t}(\UDc)$ for all $s\in(0,t)$, the norm
$\|\varphi_{s,t}-\varphi_{0,t}\|_{\dAlg}$ does not tend to zero as $s\to+0$.
\end{remark}

\subsection*{Acknowledgement}
The author would like to thank Prof.~Filippo Bracci and Prof. Manuel
D.~Contreras for important references and fruitful discussions on the topic of
this paper. The author is also grateful to Dr.~Tiziano Casavecchia for pointing
out the question on the existence of the unrestricted derivative.


\begin{thebibliography}{99}
\bibitem{Abate}M.\,Abate, {\it Iteration theory of holomorphic maps on taut manifolds},
Research and Lecture Notes in Mathematics. Complex Analysis and Geometry,
Mediterranean, Rende, 1989. MR1098711 (92i:32032)

\bibitem{MFMS_ems}M.\,Abate, F.\,Bracci, M.D.\,Contreras, and S.\,D\'{\i}az-Madrigal,
The evolution of Loewner's differential equations, Eur. Math. Soc. Newsl. No.\,78 (2010),
31--38. MR2768999

\bibitem{Aleksandrov} I.A.\,Aleksandrov,
\textit{Parametric continuations in the theory of univalent functions}, Izdat.
''Nauka``, Moscow, 1976.

\bibitem{AkhGlaz}N.\,I. Akhiezer\ and\ I.\,M. Glazman, {\it Theory of linear operators in
Hilbert space}, translated from the Russian and with a preface by Merlynd
Nestell, reprint of the 1961 and 1963 translations, Dover, New York, 1993.
MR1255973 (94i:47001)

\bibitem{BP} E.\,Berkson\ and\ H.\,Porta, Semigroups of analytic functions and composition
operators, Michigan Math. J. {\bf 25} (1978), no.~1, 101--115. MR0480965 (58
\#1112)

\bibitem{BCM1} F.\,Bracci, M.D.\,Contreras, and S.\,D\'{\i}az-Madrigal,
Evolution Families and the Loewner Equation I: the unit disk, J. reine angew.
Math. 672 (2012), 1—37, DOI 10.1515/CRELLE.2011.167. Available on
ArXiv~0807.1594

\bibitem{ClusterSets}E.F.\,Collingwood\ and\ A.J.\,Lohwater, {\it The theory of cluster sets},
Cambridge Tracts in Mathematics and Mathematical Physics, No.\,56 Cambridge
Univ. Press, Cambridge, 1966. MR0231999 (38 \#325)

\bibitem{Tiz} T.\,Casavecchia, S.\,D\'{\i}az-Madrigal, A Non-Autonomous Version of the Denjoy–Wolff
Theorem,
Complex Anal. Oper. Theory (\textsl{Online First${}^{TM}$, 2011}) DOI
10.1007/s11785-011-0214-6

\bibitem{AnaFlows} M.D.\,Contreras and S.\,D\'{\i}az-Madrigal,
Analytic flows in the unit disk: angular derivatives and
boundary fixed points, Pacific J. Math. \textbf{222} (2005), 253--286.

\bibitem{SMP} M.\,D. Contreras, S.~D\'\i az-Madrigal, P.~Gumenyuk,
Loewner chains in the unit disk, Rev. Mat. Iberoam. {\bf 26} (2010), no.~3, 975--1012.
MR2789373 (2012a:30052)

\bibitem{CMP2004} M.D.\,Contreras, S.\,D\'\i az-Madrigal\ and\ Ch.\,Pommerenke,
Fixed points and boundary behaviour of the Koenigs function, Ann. Acad. Sci.
Fenn. Math. {\bf 29} (2004), No.~2, 471--488. MR2097244 (2006b:30059)

\bibitem{CMP2006} M.D.\,Contreras, S.\,D\'\i az Madrigal\ and\ Ch.\,Pommerenke, On boundary
critical points for semigroups of analytic functions, Math. Scand. {\bf 98}
(2006), no.~1, 125--142. MR2221548 (2007b:30026)

\bibitem{parab_zoo} M.D.\,Contreras, S.\,D\'\i az-Madrigal\ and\ C.\,Pommerenke, Iteration in
the unit disk: the parabolic zoo, in {\it Complex and harmonic analysis},
63--91, DEStech Publ., Inc., Lancaster, PA, 2007. MR2387282 (2009d:30062)

\bibitem{frac1981} C.C.\,Cowen, Iteration and the solution of functional equations for
functions analytic in the unit disk, Trans. Amer. Math. Soc. {\bf 265} (1981),
no.~1, 69--95. MR0607108 (82i:30036)

\bibitem{Duren} P.L.\,Duren, {\it Univalent functions}, Springer, New York, 1983.

\bibitem{fracGor2002} M.\,Elin, V.\,Goryainov, S.\,Reich, and D.\,Shoikhet, Fractional iteration
and functional equations for functions analytic in the unit disk, Comput.
Methods Funct. Theory {\bf 2} (2002), no.~2, 353--366. MR2038126 (2005a:30046)

\bibitem{ElinShoikhetBook} M.\,Elin\ and\ D.\,Shoikhet,
{\it Linearization models for complex dynamical systems}, Operator Theory:
Advances and Applications, 208, Birkh\"auser, Basel, 2010. MR2683159
(2012c:30042)

\bibitem{ElinShoikhet} M.\,Elin\ and\ D.\,Shoikhet, Boundary behavior and rigidity of semigroups of
holomorphic mappings, Anal. Math. Phys. {\bf 1} (2011), no.~2-3, 241--258.
MR2853813 (2012i:30019)


\bibitem{Garnet}J.B.\,Garnett, {\it Bounded analytic functions}, Pure and Applied
Mathematics, 96, Academic Press, New York, 1981. MR0628971 (83g:30037)

\bibitem{Goluzin}  G.M. Goluzin, {\it Geometric theory of functions of a complex variable},
$2^{\rm nd}$\,ed., ``Nauka'', Moscow, 1966 (in Russian); English transl.: Amer.
Math. Soc., 1969.

\bibitem{SemigroupsGor}V.V.\,Gorya\u\i nov, Semigroups of conformal mappings.
(Russian), Mat. Sb. (N.S.) {\bf
129(171)} (1986), no.~4, 451--472; translation in Math. USSR Sbornik.
\textbf{57}(1987), no.\,2, 463--483. MR0842395 (87i:30018)

\bibitem{fracGor1991}V.V.\,Gorya\u\i nov, Fractional iterates of functions that
are analytic in the unit disk with given fixed points. (Russian), Mat. Sb. {\bf
182} (1991), no. 9, 1281--1299; translation in Math. USSR-Sb. {\bf 74} (1993),
no.~1, 29--46. MR1133569 (92m:30049)

\bibitem{emb1993Dokl} V.V.\,Gorya\u\i nov, The embedding of iterations of probability-generating
functions into continuous semigroups. (Russian), Dokl. Akad. Nauk {\bf 330}
(1993), no. 5, 539--541; translation in Russian Acad. Sci. Dokl. Math. {\bf 47}
(1993), no.~3, 554--557. MR1245108 (94j:60158)

\bibitem{emb1993MatSb}V.V.\,Gorya\u\i nov, Fractional iteration of probability-generating functions
and the embedding of discrete branching processes into continuous ones. (Russian),
Mat. Sb. {\bf 184} (1993), no. 5, 55--74; translation
in Russian Acad. Sci. Sb. Math. {\bf 79} (1994), no.~1, 47--61. MR1239751
(95g:60026)


\bibitem{emb1968}S.\,Karlin\ and\ J.\,McGregor, Embeddability of discrete time simple
branching processes into continuous time branching processes, Trans. Amer.
Math. Soc. {\bf 132} (1968), 115--136. MR0222966 (36 \#6015)


\bibitem{KimSugawa}Y.C.\,Kim\ and\ T.\,Sugawa, Correspondence between spirallike
functions and starlike
functions, Math. Nachr. {\bf 285} (2012), no.~2-3, 322--331. MR2881284

\bibitem{Kufarev} P.P. Kufarev, A remark on integrals of L\"owner's equation,
Doklady Akad. Nauk SSSR (N.S.) {\bf 57} (1947), 655--656. MR0023907 (9,421d)

\bibitem{Loewner}K. L\"{o}wner, Untersuchungen \"{u}ber schlichte
konforme Abbildungen des Einheitskreises, Math. Ann. \textbf{89} (1923), 103--121.

\bibitem{Poggi_Pointwise} P.\,Poggi-Corradini, Pointwise convergence on the boundary in the
Denjoy-Wolff theorem, Rocky Mountain J. Math. {\bf 40} (2010), no.~4,
1275--1288. MR2718814 (2011h:30040)

\bibitem{Pommerenke}Ch.\,Pommerenke, {\it Univalent functions}.
With a chapter on quadratic differentials by Gerd Jensen, Vandenhoeck \&
Ruprecht, G\"{o}ttingen, 1975.

\bibitem{Pommerenke-II}Ch. Pommerenke, {\it Boundary behaviour of conformal maps},
Grundlehren der Mathematischen Wissenschaften, 299, Springer,
Berlin, 1992. MR1217706 (95b:30008)

\bibitem{ShoikhetBook}D.\,Shoikhet, {\it Semigroups in geometrical function theory}, Kluwer Acad.
Publ., Dordrecht, 2001. MR1849612 (2002g:30012)

\bibitem{SiskakisReview} A.\,G. Siskakis, Semigroups of composition operators on spaces of analytic
functions, a review, in {\it Studies on composition operators (Laramie, WY,
1996)}, 229--252, Contemp. Math., 213 Amer. Math. Soc., Providence, RI.
MR1601120 (98m:47049)

\bibitem{frac1958}G. Szekeres, Regular iteration of real and complex functions, Acta Math.
{\bf 100} (1958), 203--258. MR0107016 (21 \#5744)
\end{thebibliography}
\end{document}